\documentclass[reqno,10pt,a4paper]{amsart}
\usepackage[margin=2.7cm]{geometry}
\usepackage{amsthm}
\usepackage{array}
\usepackage{multirow}
\usepackage{pdfpages}
\usepackage{cite}
\usepackage{booktabs}
\usepackage{realboxes}
\usepackage{orcidlink}
\usepackage{enumitem}
\usepackage{mathtools}
\usepackage{upgreek}
\usepackage{mathrsfs}
\usepackage{xcolor}
\usepackage{hyperref}
\hypersetup{
    colorlinks,
    linkcolor={red!50!black},
    citecolor={green!50!black},
    urlcolor={blue!80!black}
}
\newcommand{\cB}{\mathscr B}
\newcommand\grdb{{\sc Grdb}}
\newcommand{\N}{\mathbb N}
\newcommand{\PP}{\mathbb P}
\newcommand{\Q}{\mathbb Q}
\newcommand{\F}{\mathbb F}
\newcommand{\C}{\mathbb C}
\newcommand{\BB}{\mathbb B}
\newcommand{\Z}{\mathbb Z}
\newcommand{\cO}{\mathscr O}
\newcommand{\cE}{{\mathcal E}}
\newcommand{\cF}{{\mathcal F}}
\newcommand{\FMF}{{\mathcal F}_{4}}
\newcommand{\Fss}{{\mathcal F}_{3}}
\newcommand{\Tcan}{{\mathcal T}_{\mathrm{can}}}
\newcommand{\gmin}{g_{\mathrm{min}}}
\newcommand{\gmax}{g_{\mathrm{max}}}
\newcommand{\tf}{$3$\nobreakdash-fold}
\newcommand{\tfs}{$3$\nobreakdash-folds}
\newcommand{\wps}{w\PP}

\newcommand{\bmu}{\boldsymbol\mu}

\newcommand{\ep}{\varepsilon}
\newcommand{\MM}[2]{#1\text{--}#2} 
\DeclareMathOperator{\Proj}{Proj}
\DeclareMathOperator{\Bl}{Bl}
\DeclareMathOperator{\Grass}{Grass}
\newcommand{\wGrass}{w\!\Grass}
\DeclareMathOperator{\Cl}{Cl}
\DeclareMathOperator{\TV}{TorVar}
\DeclareMathOperator{\Sym}{Sym}
\newcommand{\De}{\Delta}
\newcommand{\cT}{\mathcal{T}}
\newcolumntype{C}{>{\centering\arraybackslash}p{0.5cm}}
\newcolumntype{W}{>{\centering\arraybackslash}p{0.85cm}}
\newcommand{\rotateplot}[2]{%
\Resizebox{#1}{!}{\Rotatebox{90}{%
\renewcommand{\arraycolsep}{1.5pt}%
\renewcommand{\arraystretch}{1.2}%
#2%
}}%
}
\renewcommand{\kappa}{\upkappa}

\AtBeginDocument{%
\def\MR#1{}
}
\newcommand\blfootnote[1]{%
  \begingroup
  \renewcommand\thefootnote{}\footnote{#1}%
  \addtocounter{footnote}{-1}%
  \endgroup
}
\newtheorem{thm}{Theorem}[section]
\newtheorem{prop}[thm]{Proposition}
\newtheorem{lemma}[thm]{Lemma}
\newtheorem{cor}[thm]{Corollary}
\theoremstyle{definition}
\newtheorem{defn}[thm]{Definition}
\newtheorem{example}[thm]{Example}
\newtheorem{remark}[thm]{Remark}
\newtheorem{ansatz}[thm]{Ansatz}
\numberwithin{equation}{section}

\begin{document}
\author[G.\,Brown]{Gavin Brown\,\orcidlink{0000-0002-4087-5624}}
\author[A.\,M.\,Kasprzyk]{Alexander M.\ Kasprzyk\,\orcidlink{0000-0003-2340-5257}}
\author[M.\,Reid]{Miles Reid\,\orcidlink{0000-0001-6299-8804}}
\address{Mathematics Institute, Zeeman Building, University of Warwick, Coventry, CV4 7AL, UK}
\email{G.Brown@warwick.ac.uk}
\email{Alexander.Kasprzyk@warwick.ac.uk}
\email{Miles.Reid@warwick.ac.uk}
\title[The Graded Ring Database]{Kawamata bounds for Fano threefolds\\and the Graded Ring Database}
\begin{abstract}
We give an effective realisation of Kawamata's famous bound for Fano \tfs. Our result is a list of~$39,\!550$ candidate Hilbert series $P$ for Fano \tfs\ satisfying the Bogomolov stability inequality.  For each $P$, there may or may not be a Fano \tf\ $X$ that realises $P$, but even so we give a set of weights that we expect to be among those of a minimal generating set of the anticanonical ring of any such~$X$. This provides a crude first approximation to the ongoing classification of Fano \tfs. We explain in detail what this data means, the important nuances that current known examples reveal about it, and how it relates to a range of different classification programmes.
\end{abstract}
\maketitle
\blfootnote{\emph{Nagoya Mathematical Journal} (2026), \textbf{261}:e41 1--32, doi:\href{https://doi.org/10.1017/nmj.2026.10122}{10.1017/nmj.2026.10122}}
\section{Introduction}
The~\emph{Graded Ring Database}, or~\grdb, is a collection of data
that relates to polarised varieties and their graded rings. The motivation
for such a database is founded on various celebrated finiteness results
for Fano varieties \cite{kawamata,kmmt,birkar}. Fano varieties are among
the key elementary particles of geometry, and here we are interested
specifically in Fano~\tfs\ over $\C$. Nonsingular Fano~\tfs\ are classified
into 105 families \cite{I2,MM}, and are recorded with additional data
both in print \cite{IP} and online \cite{fanography}. The nonsingular
Fanos form a kind of aristocracy, making them a popular topic for
certain authors. However, the broad toiling masses of Fano~\tfs\ with
terminal singularities are not to be neglected, and their study offers
rich rewards. Several hundred families of them are already known~(see~\S\ref{sec!known}), with the reasonable expectation of thousands or
tens of thousands more yet to be discovered -- and orders of magnitude
more if we relax the conditions on singularities further. Ultimately,
if ever completed, the classification will be a computer database rather
than a physical atlas. The \grdb\ is a first approximation to that
database. In the meantime it provides a roadmap for exploring the
landscape, and for identifying candidates for construction or elimination.

It is a key point that, while not huge, the data is too large to be written down usefully, and alternative access is vital.
The easiest way to access and interrogate the Fano~\tf\ database (the data described precisely in \ref{def!Fss}) is online at~\cite{grdbweb}; some form of this data access has been online for over 25 years since \cite{ABR02}.
The underlying raw data is available at~\cite{fanodata} in generic SQL format.
Code to generate the data is at~\cite{grdbweb} (for Magma \cite{magma}) and \cite{fanocode} (for Go \cite{Go}).

Two main considerations govern our view of Fano 3-folds (or $\Q$-Fano
3-folds) and the precise terminology we use here.
Prime Fano \tfs\ (see Definition~\ref{def!prime}) form the strictest class.
They have a distinguished place in the birational classification of \tfs:
each one is the endpoint of a Minimal Model Program~(MMP) in the Mori category \cite[Theorem~0.3.1]{mori88}.
On the other hand, the main numerical invariant of a Fano variety that we use is its
anticanonical Hilbert series (\S\ref{sec!fano}) and the ingredients that go into it (\S\ref{sec!numerical}).
This is the basic object of the \grdb\ and it works for a much broader definition
allowing canonical singularities (see Definition~\ref{def!fano}).
For varieties with worse than canonical singularities this viewpoint
breaks down, even for mild log terminal orbifold points.
Throughout this paper,  we say {\it {\tf}} to mean a normal variety of dimension~3 over~$\C$.

\begin{defn}\label{def!fano}
A~\emph{Fano~\tf\ with canonical singularities} is a projective \tf\ $X$
with at worst canonical singularities, for which the
anticanonical class $-K_X$ is ample. 
We usually abbreviate `Fano~\tf\ with canonical singularities' to `Fano~\tf'. By
definition, the~\emph{genus} of~$X$ is~$g_X\coloneq  h^0(X,-K_X)-2$.
\end{defn}

In \cite{kmmt} this class is called {\em canonical $\Q$-Fano \tfs} and
it is shown to be bounded (ibid.\ 1.3). The point of the~\grdb\ is to
realise this bound explicitly under stated conditions.
Our main result is a list~$\Fss$ of rational functions (defined in Definition~\ref{def!Fss})
that satisfy necessary conditions to be the Hilbert series of a Fano \tf\
(polarised by~$-K_X$ as in~\S\ref{sec!numerical}), subject to an additional numerical condition.

\begin{thm}
\label{thm!main}
Let~$X$ be a Fano~\tf\ satisfying $-K_X^3\le -3K_Xc_2$. Then its Hilbert
series~$P_X(t)$ is one of the~$39,\!550$ rational functions in the
list~$\Fss$.
\end{thm}

The condition $-K_X^3\le -3K_Xc_2$ is natural, as we discuss further below. 
Our use of it is motivated by the proof of Kawamata's boundedness
theorem~\cite{kawamata}, one of the earliest for \tfs,
and one which seems well suited to explicit classification in a range of contexts.

\begin{defn}[Prime Fano 3-fold]\label{def!prime}
The {\em Mori category} consists of $3$-folds $X$ with $\Q$-factorial
terminal singularities. We say a 3-fold $X$ in the Mori category is a
{\em prime Fano \tf} if $-K_X$ is ample, and the Weil divisor class
group $\Cl X$ has rank $\rho_X=1$.
\end{defn}

Here we intend $\Q$-factorial in the algebraic sense: the
local ring $\cO_{X,P}$ at every $P\in X$ is $\Q$-factorial, or
equivalently, any algebraic Weil divisor on $X$ has a multiple that is
Cartier. The complex analytic variety $X^{\mathrm{an}}$ may not be $\Q$-factorial in the
local or \'{e}tale analytic setting: as an example, a cubic or quartic
3-fold hypersurface $X\subset\PP^4$ with a single ordinary node is
$\Q$-factorial, although the node is (analytic, \'{e}tale or formal)-locally isomorphic to $xy=zt$.

On the other hand, we do not insist that $-K_X$ generates the class group
$\Cl X$: we allow $-K_X\in\Cl X$ to be divisible (see Definition~\ref{def!index} below for a
discussion of index), and $\Cl X$ may have a finite torsion group.

Prime Fano~\tfs\ remain unclassified.
The proof of~\cite[Proposition~1]{kawamata} determines inequalities on prime Fano \tfs\ of the form
${-}K_X^3\le {-}\kappa K_Xc_2(X)$ for certain values of~$\kappa>0$.
These in turn impose numerical constraints on the Hilbert series~$P_{X}$ of~$X$. 
A Fano~\tf\ $X$ is Bogomolov~\emph{semistable} if the
sheaf~$\cE=({\Omega^1_X})^{\vee\vee}$ is semistable. Kawamata shows
that for prime semistable $X$,
the inequality above holds with~$\kappa=3$ as in Theorem~\ref{thm!main};
it follows at once from the Bogomolov--Miyaoka--Yau inequality $c_1^2\le3c_2$
applied to $\cE$ restricted to suitable $S\in|{-}RK_X|$ for large divisible~$R$.

More recently Liu and Liu \cite[Theorem~1.3]{LL25} prove that the strict inequality ${-}K_X^3 < -3K_Xc_2(X)$ holds for any prime Fano \tf\ with terminal
singularities (cf.\ \cite{S24}), and so Theorem~\ref{thm!main} applies in this case.
Only 208 of the 39,550 candidates fail this strict inequality, and none of these 208 is currently realised.

\begin{cor}\label{cor!ss}
If $X$ is a prime Fano~\tf, then $P_X\in\Fss$. 
\end{cor}

\begin{table}[h!]
\centering
\rotateplot{0.73\linewidth}{
\begin{tabular}{cc|CCCCC|CCCCC|CCCCC|CCCCC|CCCCC|CCCCC|CCCCC|CCCCC|r}
\multicolumn{2}{c}{}&\multicolumn{40}{c}{Genus}\\
&&$-2$&$-1$&0&1&2&3&4&5&6&7&8&9&10&11&12&13&14&15&16&17&18&19&20&21&22&23&24&25&26&27&28&29&30&31&32&33&34&35&36&37&\textrm{total}\\
\cline{2-43}
\multirow{35}*{\rotatebox{-90}{Codimension}}&1&&54&32&6&2&1&&&&&&&&&&&&&&&&&&&&&&&&&&&&&&&&&&&95\\
&2&1&45&29&6&2&1&1&&&&&&&&&&&&&&&&&&&&&&&&&&&&&&&&&&85\\
&3&&26&29&8&3&2&1&1&&&&&&&&&&&&&&&&&&&&&&&&&&&&&&&&&70\\
&4&3&60&54&15&6&3&2&1&1&&&&&&&&&&&&&&&&&&&&&&&&&&&&&&&&145\\
&5&1&58&63&21&8&6&3&2&1&1&&&&&&&&&&&&&&&&&&&&&&&&&&&&&&&164\\
&6&4&80&98&35&15&8&6&3&2&1&1&&&&&&&&&&&&&&&&&&&&&&&&&&&&&&253\\
&7&&81&116&49&21&15&8&6&3&2&1&1&&&&&&&&&&&&&&&&&&&&&&&&&&&&&303\\
\cline{2-43}
&8&6&128&182&79&35&21&15&8&6&3&2&1&1&&&&&&&&&&&&&&&&&&&&&&&&&&&&487\\
&9&3&107&208&109&49&35&21&15&8&6&3&2&1&1&&&&&&&&&&&&&&&&&&&&&&&&&&&568\\
&10&10&194&314&171&79&49&35&21&15&8&6&3&2&1&1&&&&&&&&&&&&&&&&&&&&&&&&&&909\\
&11&5&168&378&236&109&79&49&35&21&15&8&6&3&2&1&1&&&&&&&&&&&&&&&&&&&&&&&&&1116\\
&12&7&244&538&353&171&109&79&49&35&21&15&8&6&3&2&1&1&&&&&&&&&&&&&&&&&&&&&&&&1642\\
&13&8&254&638&485&236&171&109&79&49&35&21&15&8&6&3&2&1&1&&&&&&&&&&&&&&&&&&&&&&&2121\\
&14&11&377&879&698&353&236&171&109&79&49&35&21&15&8&6&3&2&1&1&&&&&&&&&&&&&&&&&&&&&&3054\\
\cline{2-43}
&15&10&349&889&870&480&351&236&171&109&79&49&35&21&15&8&6&3&2&1&1&&&&&&&&&&&&&&&&&&&&&3685\\
&16&15&509&802&998&675&473&349&235&171&109&79&49&35&21&15&8&6&3&2&1&1&&&&&&&&&&&&&&&&&&&&4556\\
&17&21&501&480&740&786&637&456&342&233&171&109&79&49&35&21&15&8&6&3&2&1&1&&&&&&&&&&&&&&&&&&&4696\\
&18&17&604&284&318&773&658&571&427&331&228&169&109&79&49&35&21&15&8&6&3&2&1&1&&&&&&&&&&&&&&&&&&4709\\
&19&23&465&103&43&401&516&488&474&380&308&221&167&108&79&49&35&21&15&8&6&3&2&1&1&&&&&&&&&&&&&&&&&3917\\
&20&31&220&44&4&106&159&267&299&345&308&272&204&160&106&79&49&35&21&15&8&6&3&2&1&1&&&&&&&&&&&&&&&&2745\\
&21&32&3&5&&8&24&42&101&137&202&213&213&176&149&101&77&49&35&21&15&8&6&3&2&1&1&&&&&&&&&&&&&&&1624\\
\cline{2-43}
&22&27&&1&&1&1&3&7&28&43&88&117&141&135&128&94&75&48&35&21&15&8&6&3&2&1&1&&&&&&&&&&&&&&1029\\
&23&13&&&&&&&1&1&5&7&25&43&71&84&98&78&68&46&35&21&15&8&6&3&2&1&1&&&&&&&&&&&&&632\\
&24&10&&&&&&&&&&1&1&5&10&23&36&59&54&57&41&33&21&15&8&6&3&2&1&1&&&&&&&&&&&&387\\
&25&4&&&&&&&&&&&&&1&1&4&10&25&27&39&35&31&20&15&8&6&3&2&1&1&&&&&&&&&&&233\\
&26&2&&&&&&&&&&&&&&&&1&1&7&9&20&23&24&18&15&8&6&3&2&1&1&&&&&&&&&&141\\
&27&&&&&&&&&&&&&&&&&&&&1&1&6&10&15&13&13&8&6&3&2&1&1&&&&&&&&&80\\
&28&&&&&&&&&&&&&&&&&&&&&&&1&2&4&8&11&7&6&3&2&1&1&&&&&&&&46\\
\cline{2-43}
&29&&&&&&&&&&&&&&&&&&&&&&&&&&1&2&6&5&6&3&2&1&1&&&&&&&27\\
&30&&&&&&&&&&&&&&&&&&&&&&&&&&&&&2&1&4&3&2&1&1&&&&&&14\\
&31&&&&&&&&&&&&&&&&&&&&&&&&&&&&&&&&2&2&2&1&1&&&&&8\\
&32&&&&&&&&&&&&&&&&&&&&&&&&&&&&&&&&&&1&2&1&1&&&&5\\
&33&&&&&&&&&&&&&&&&&&&&&&&&&&&&&&&&&&&&&1&1&&&2\\
&34&&&&&&&&&&&&&&&&&&&&&&&&&&&&&&&&&&&&&&&1&&1\\
&35&&&&&&&&&&&&&&&&&&&&&&&&&&&&&&&&&&&&&&&&1&1\\
\cline{2-43}
\multicolumn{2}{r|}{\textrm{total}}&264&4527&6166&5244&4319&3555&2912&2386&1955&1594&1300&1056&853&692&557&450&364&288&229&182&146&117&91&71&53&43&34&26&20&14&11&9&6&5&4&2&2&1&1&1&39550
\end{tabular}
}
\caption{Number of semistable Fano~$3$-fold Hilbert series in~$\Fss$, listed by
genus and estimated minimal codimension.\label{fig!geography}}
\end{table}

In fact, $\cF_3$ contains almost all Fano~\tfs\ we know.
Most series in $\cF_3$ are not the Hilbert series of a prime Fano \tf: the
main issue is to begin to understand which ones are, which ones are realised by
more general classes of Fano \tf, and which ones are not realised by Fano \tfs\ at all.

Relaxing~$\kappa=3$ to~$\kappa=4$ accommodates most of Kawamata's
conditions even in the non-semistable case (see Theorem~\ref{thm!kaw}). Imposing only
that weaker condition gives a larger list~$\FMF$ of~$52,\!646$ rational
functions that contains~$\Fss$  (see Proposition~\ref{prop!list}).
Liu and Liu \cite[Theorem~1.1]{LL24} also prove that if $X$ is a Fano \tf\ with $\rho_X=1$ which is nonsingular
in codimension~2, then it satisfies the strict inequality $-K_X^3 < -4K_Xc_2$. 
Only 87 of the $52,\!646$ fail this strict inequality.

\begin{cor}\label{cor!kappa4}
If $X$ is a Fano~\tf\ with isolated canonical singularities and $\rho_X=1$,
then $P_X\in\FMF$.
\end{cor}

For example, the weighted projective space~$\PP(1,1,3,5)$ lies in $\FMF\setminus\Fss$.
It has Picard rank $\rho=1$, but is not prime because it has the mild canonical singularity $\tfrac{1}5(1,1,3)$.

\begin{defn}[cf.~Definition~\ref{def!Fss}]
The \emph{Fano~\tf\ database} is the pair of lists $\Fss\subset\FMF$.
We say that a Fano~\tf~$X$ is in~$\Fss$ when~$P_X\in\Fss$.
The Fano~\tf\ database may be accessed online~\cite{grdbweb} and the raw
data is available at~\cite{fanodata} under a~CC0 licence~\cite{CC0}.
\end{defn}

The \grdb\ also includes the data of all toric Fano \tfs, so it is easy to derive the following.
\begin{prop}[\S\ref{sec!toric} and Table \ref{fig!cantoricnumbers}]
The $674,\!688$ toric Fano \tfs\ realise $5772$ different Hilbert series.
If $X$ is a toric Fano \tf, then $P_X\in\FMF$, and $5760$ of these Hilbert series lie in $\Fss$.
\end{prop}

We emphasise that some unknown, and possibly significant, fraction 
of the elements of $\Fss$ cannot be realised as the Hilbert
series of a  Fano~\tf; see~\S\ref{sec!warning} for a recapitulation
of this point and~\S\ref{sec!examples} for several other possible
confusions. 

We explain the proof of Theorem~\ref{thm!main} in \S\ref{sec!proof}. This may be thought
of as folklore, and similar in spirit to Prokhorov's degree bound
computation~\cite[(1.2)]{prokhorov07}. Nevertheless the steps turning
Kawamata's theorem into a concrete list are important in light of
the bounds on Fano varieties more generally. The main
results of this paper supplement~\ref{thm!main} as follows:

\begin{enumerate}
\item
For each~$P\in\Fss$, we estimate the smallest anticanonical embedding that a
Fano~\tf~$X$ with Hilbert series~$P_X=P$ could have
(\S\ref{sec!geography}), and we use those to present~$\Fss$ as a geographical
map of candidates~(Table~\ref{fig!geography}). This geography, and its meaning as the basis
for a programme of classification, is the main result of this paper.
\item
We identify some known classifications within the Fano~\tf\ database
(\S\ref{sec!known}) and compare with some well-known results
(\S\ref{sec!corollaries}).
\item
We highlight possible misunderstandings of the Fano~\tf\ database
(\S\S\ref{sec!examples} and~\ref{sec!warning}).
\item
The numbers of cases~$\#\Fss=39,\!550$ and~$\#\FMF=52,\!646$ arise from
elementary combinatorial considerations (Proposition~\ref{prop!list}).
\end{enumerate}

In outline, \S\ref{sec!building} constructs the lists $\Fss$ and $\FMF$,
\S\ref{sec!geography} explains the estimates of codimension that lead to~Table~\ref{fig!geography},
\S\ref{sec!known} describes a number of constructions and comparisons with known lists of
Fano \tfs, and finally \S\ref{sec!summary} summarises the \grdb\ together with important warnings about its meaning.

We put some emphasis on toric Fano~\tfs\ throughout, in part to profit from
simultaneous use of the Fano~\tf\ database and the toric Fano classification in
the~\grdb. We indicate in \S\ref{sec!toric} some different ways that this
interdisciplinary approach may yet be exploited.

The systematic study of the graded rings of all Fano \tfs\ together in
this way was initiated by Miles Reid; see~\cite{ABR02}, which
alongside~\cite{fletcher,altinok} was the starting point for the~\grdb.
The \grdb\ has been available since then in different forms
\cite{grdbweb,fanodata}, though without the documentation and detailed
analysis this paper provides. The need for this analysis has become more
important recently, as a series of developments promise to expose large
new parts of the classification:
\begin{itemize}
\item
The use of formats by Takagi \cite{takagi23a,takagi24d,takagi25}
and Coughlan--Ducat \cite{CD}
extends the cases of \cite{BKQ,formats} (\S\ref{sec!formats})
and opens up mid-range codimension to similar analysis;
\item
New Kawamata--Miyaoka inequalities \cite{LL24}
provide sharper data in broader contexts (cf.\ Corollaries~\ref{cor!ss} and~\ref{cor!kappa4}) and also higher
dimension;
\item
Variations of the Papadakis--Reid unprojection techniques
exploited in \cite{tj} provide constructions for
many more families \cite{P22,BCD1,BCD2,Taylor};
\item
The Fanosearch programme \cite{CCGGK13,dialogues}
(\S\ref{sec!fanosearch}) attacks classification via Mirror
Symmetry, exploiting certain types of degeneration; thus as a first
approximation the raw \grdb\ list is supplemented by matching toric Fano
\tfs, which are classified (\S\ref{sec!toric}).
\end{itemize}
This artcle makes the \grdb\ precise: it provides the theorems that underpin the data
and its construction,
it relates the data to a range of classical results, and by locating a
class of toric varieties within the data it provides a first step in identifying certain degenerations.
There is a quick startup guide in~\S\ref{sec!summary}, which also summarises the essential warnings for its use.

The beauty and enduring interest of Fano classification lies in the individual
\tfs\ and deformation families we meet, rather than the bureaucracy of
compartmentalising them. A map is only a map: the actual adventures happen out
in the field. The real value of the geography in Table~\ref{fig!geography}
is to identify hundreds of wonderful places to explore: compare \cite{tj,CHNP,suess,DNFF,CD,takagi24a,Heu}.

\section{Building the Fano 3-fold database}\label{sec!building}
We recall standard material related to the plurigenus formula and use it to
assemble the data that proves Theorem~\ref{thm!main} (cf.\ \cite[\S4]{ABR02}).

\subsection{Fano 3-folds}\label{sec!fano}
Any Fano~\tf~$X$ comes with an intrinsic embedding~$X\subset\wps$ in weighted
projective space (up to automorphisms of~$\wps$) as follows. The~\emph{graded
ring} of~$X$ is
\[
R(X,-K_X) = \bigoplus_{m\in\N} H^0(X,-mK_X)
\]
and~$X\cong \Proj R(X,-K_X)$ by ampleness. Any choice~$f_0,\dots,f_n$ of minimal
homogeneous generators for~$R(X,-K_X)$ has the same collection of
weights~$\{a_0,\dots,a_n\}$, where~$a_i=\deg f_i$, and we may
suppose~$a_0\le\cdots\le a_n$. (Note the conventional abuse of
notation:~$\{a_0,\dots,a_n\}$ is a list with possible repetitions, and similarly
$\{a_0,\dots,a_n\}\setminus\{a_3,a_7\}$
is the list obtained by removing one instance of each of~$a_3$ and~$a_7$.)
Thus any choice of minimal generating set determines a projectively normal embedding
\begin{equation}\label{eq!acemb}
X \subset \PP(a_0,\dots,a_n)
\end{equation}
which we refer to as the~\emph{anticanonical embedding of~$X$}.
(This is not the same as the image~$\Phi_{-K_X}(X)$ under the linear
system~$|{-}K_X|$ unless~$-K_X$ is very ample, which is rarely the case here.)

The~\emph{Hilbert series} $P_X$ of a Fano~\tf~$X$ is simply that of~$R(X,-K_X)$,
the formal power series
\[
P_X = P_X(t) = \sum_{m\in\N} h^0(X,-mK_X) t^m.
\]
The Hilbert series~$P_X$ of a Fano~\tf~$X$ does
not determine the weights~$a_i$ of its anticanonical embedding -- it not
infrequently happens that $R(X,-K_X)$ has relations and generators in the same degree.

\subsection{Numerical data of a Fano 3-fold}\label{sec!numerical}
We write $\tfrac{1}r(a_1,a_2,a_3)$ to denote the singularity germ of the origin in $\C^3/\bmu_r$, where $\bmu_r$ is the multiplicative group of $r$th roots of unity,
and $\ep\in\bmu_r$ acts on coordinates by $(\ep^{a_1}x_1,\ep^{a_2}x_2,\ep^{a_3}x_3)$.
Any Fano~\tf\ has a~\emph{basket of singularities}~$\cB$, which is a collection
of terminal quotient singularities~$\tfrac{1}{r}(1,a,-a)$ (possibly including
repeats, where again we use set notation on this understanding).
Following~\cite[\S\S8\nobreakdash--10]{ypg}, in general~$\cB$ is derived locally
from the singularities of~$X$ by crepant blowup and Q\nobreakdash-smoothing.
When~$X$ has terminal singularities and lies in weighted projective
space as a quasi\-smooth variety, then $X$ is an orbifold and~$\cB$ is exactly
the collection of singularities of~$X$.

The~\emph{genus}~$g=h^0(X,-K_X)-2\ge -2$ 
and basket~$\cB$ together determine the degree by the formula
\begin{equation}
\label{eq!K3}
{-}K_X^3 = 2g-2 + \sum_\cB \frac{b(r-b)}{r}
\end{equation}
where for each element~$\tfrac{1}{r}(1,a,-a)$ of the basket~$\cB$, we define $0<b<r$
by~$ab\equiv 1\mod r$.

Although they are not used for the initial construction of the \grdb, there are
various different divisorial indices defined for a Fano~\tf~$X$ that we consider
later.

\begin{defn}\label{def!index}
The~\emph{Gorenstein index} or~\emph{singularity index} of~$X$ is
\[
i_X = \min \{ n\in\Z_{>0} \mid nK_X \text{ is Cartier}\}.
\]
The~\emph{Fano index} is the rational number~$f_X = r/i_X$ where $r>0$ is the largest integer for
which~$-i_XK_X \equiv rA$ for a Cartier divisor~$A$. There are further types of
Fano index:
\begin{align*}
q_X&= \max \{ q>0 \mid {-}K_X \sim qA, \text{ where~$A$ is a~$\Q$\nobreakdash-Cartier Weil divisor} \}\\
q_{\Q X}&= \max \{ q>0 \mid {-}K_X \sim_\Q qA, \text{ where~$A$ is a~$\Q$\nobreakdash-Cartier Weil divisor} \}.
\end{align*}
\end{defn}

These are natural generalisations of the divisibility of the anticanonical
divisor of a nonsingular Fano~\tf\ in its Picard group. For a Fano~\tf, both~$q$
and~$q_\Q$ are positive integers, and they are equal if~$\Cl(X)$ has no
torsion~\cite{prokhorov10}.

\subsection{Effective Kawamata bounds for prime Fano 3-folds}
We construct the set of genus--basket pairs~$(g,\cB)$ that satisfy the constraints
of~\cite{kawamata}.

\subsubsection{Possible baskets}
Controlling the possible baskets of Fano~\tfs\ can be done much more generally
than the prime Fano case. Recall that a basket $\cB$ is a set with possible repeats.

\begin{defn}\label{def!basket}
We fix a particular set $\BB$ of baskets $\cB$ as
$\BB = \left\{
\cB \mid \sum_{\cB} (r - \frac{1}{r}) < 24
\right\}$,
where each $\cB$ is a basket and the sum is taken over the quotient singularities $\tfrac{1}r(1,a,-a)\in\cB$.
\end{defn}
It is elementary and easy to check that $\#\BB = 8314$.
For any $\cB\in\BB$, the defining bound of 24 implies first that $r\le24$ for
each~$\tfrac{1}r(1,a,-a)\in\cB$, and moreover that $\#\cB\le 15$, since each $r-(1/r)\ge 3/2$.
Therefore
there are only finitely many possible collections of indices~$r$ appearing
in~$\cB$.
The only terminal quotient singularity with~$r=2$ is~$\tfrac{1}2(1,1,1)$, and for
each index~$r>2$, there are~$\phi(r)/2$ distinct terminal quotient singularities,
namely~$\tfrac{1}r(1,a,-a)$ for~$1\le a< r/2$ coprime to~$r$. Thus for each
index~$r$ that appears in~$\cB$, there are only finitely many singularities of
index~$r$ that may occur. The precise bound~8314 comes by enumerating all possible baskets satisfying these
conditions (by computer, for example).

\begin{thm}[\!{\cite[Lemmas~2.2 and~2.3]{kawamata86},~\cite[(10.3)]{ypg}}]
Let~$X$ be a projective \tf\ with canonical singularities. Then
\begin{equation}
\label{eq!kc2}
24\chi(\cO_X) = {-}K_Xc_2(X) + \sum_{\cB} \left(r - \frac{1}r\right)
\end{equation}
where the sum is over the singularities~$\tfrac{1}r(1,a,-a)$ of the
basket~$\cB$ of~$X$.

When~$X$ is a Fano~\tf,~$\chi(\cO_X)=1$ and~\eqref{eq!kc2} simplifies.
\end{thm}

\begin{lemma}\label{lem!baskets}
If~$X$ is
a Fano~\tf\ with canonical singularities which satisfies $-K_X c_2(X) > 0$,
then the basket~$\cB$ of~$X$ is in~$\BB$.
\end{lemma}
\begin{proof}
The condition~$-K_X c_2(X) > 0$ together with~\eqref{eq!kc2} implies
that~$\sum r - (1/r) < 24$, where the sum is taken over the basket~$\cB$. 
\end{proof}

\begin{remark}
The condition~$-K_Xc_2(X)>0$ holds for prime Fano~\tfs\ by Theorem~\ref{thm!kaw}. Far
more generally, any weak Fano~\tf\ (i.e.,~$-K_X$ is only required to be nef
and big) with terminal singularities satisfies~$-K_Xc_2(X)\ge0$
by~\cite[1.2(1)]{kmmt}. It is easy to check that relaxing the inequality
provides (coincidentally) 24 additional cases with~$\sum r-(1/r)=24$,
such as~$16\times\tfrac{1}2(1,1,1)$ and $5\times\tfrac{1}5(1,2,3)$. 
\end{remark}

\subsubsection{Genus bounds}\label{sec!genus}
Controlling the possible values for the genus~$g_X$ for each basket, uses the
bounds determined by~\cite{kawamata}.

\begin{prop}\label{cor!g}
Let~$X$ be a Fano~\tf\ with basket~$\cB$ and genus~$g$ that
satisfies
\begin{equation}\label{eq!kaw}
{-}K_X^3 \le \kappa (-K_Xc_2(X))
\end{equation}
for some $\kappa>0$.
Then~$\gmin \le g \le \gmax$ where
\[
\gmin = \max\left\{ -2, \left\lfloor \tfrac{1}2 \left(2 - \sum_{\cB} \tfrac{b(r-b)}r \right) \right\rfloor+1\right\}
\]
and
\begin{equation}
\label{eq:gmax}
\gmax = \left\lfloor \tfrac{1}2 \left(2 - \sum_{\cB} \tfrac{b(r-b)}r +
\kappa \left(24 - \sum_{\cB} \left(r - \tfrac{1}r\right) \right)\right) \right\rfloor
\end{equation}
where each sum is over~$\tfrac{1}r(1,a,-a)\in\cB$ and~$b$ is defined
by~$ab\equiv 1\mod r$.
\end{prop}

\begin{proof}
The lower bound is simply the condition that~$-K_X^3>0$ in~\eqref{eq!K3}. For
the upper bound, substituting~\eqref{eq!K3} and~\eqref{eq!kc2}
into~\eqref{eq!kaw} gives
\[
2g-2 + \sum_\cB \tfrac{b(r-b)}{r} \le \kappa\left(24 - \sum_{\cB} \left(r - \tfrac{1}r\right)\right)
\]
and the upper bound follows.
\end{proof}

\begin{thm}[Kawamata~\cite{kawamata}]\label{thm!kaw}
Let~$X$ be a prime Fano~\tf. Then
\[
{-}K_X^3 \le \kappa (-K_Xc_2(X))
\]
for some real number~$\kappa>0$. In particular,~$-K_X c_2(X) > 0$.

If~$X$ is semistable, then the formula~\eqref{eq:gmax} holds with~$\kappa=3$.
If~$(\Omega_X^1)^{\vee\vee}$ has a rank~$2$ maximal destabilising subsheaf,
then the formula holds with~$\kappa=4$. If~$(\Omega_X^1)^{\vee\vee}$ has a
rank~$1$ maximal destabilising subsheaf, then it holds with possibly
larger~$\kappa>0$.
\end{thm}

\subsection{Proof of~\ref{thm!main}}\label{sec!proof}
The Hilbert series~$P_X$ is equivalent to the data of genus--basket
pair~$(g,\cB)$ by the following Fletcher--Reid plurigenus formula together
with~\eqref{eq!K3} and the independence of basket
contributions~\cite[4.2]{fletcherinvert}.
We could instead use the integral orbifold Riemann--Roch formulas of
\cite{BRZ}, and for human calculation they are much preferred.

\begin{thm}[\!{\cite[Theorem~2.5]{plurigenus},~\cite[(10.3)]{ypg}}]
\label{thm:plurigenus}
Let~$X$ be a Fano~\tf\ with canonical singularities and basket~$\cB$. Then
\begin{equation}\label{eq!RR}
 P_X(t) = \frac{1+t}{(1-t)^2} - \frac{t(1+t)}{(1-t)^4}\,\frac{K_X^3}2
 -\sum_\cB \left(\frac{1}{(1-t)(1-t^r)}
 \sum\limits_{i=1}^{r-1}
 \frac{\overline{bi}(r-\overline{bi})t^i}{2r}\right)
\end{equation}
where for each element $\tfrac{1}{r}(1,a,-a)$ of the basket~$\cB$, we fix any integer~$b$
that satisfies $ab\equiv 1\mod r$, and set $\overline{bi}\in\{0,1,\ldots,r-1\}$ to be the least
residue of $(bi)$ modulo~$r$.
\end{thm}

\begin{defn}\label{def!Fss}
For any~$g\ge-2$ and~$\cB\in\BB$, the right-hand side of~\eqref{eq!RR}
determines a rational function that we denote~$P_{g,\cB}$, where~\eqref{eq!K3} expresses the $-K_X^3$ factor in terms of $g,\cB$.
We may expand $P_{g,\cB} = \sum P_nt^n$ as a power series, and conversely given such a power series we may recover $g,\cB$
and the expression~\eqref{eq!RR} as a rational function.
For any $\kappa>0$, define
\[
\cF_\kappa = \left\{
(g,\cB,P_{g,\cB}) \mid \cB\in\BB \textrm{ and } \gmin\le g\le\gmax, \textrm{ and if $P_1>0$ then $P_n>0$ for all $n\ge1$}
\right\}
\]
where $\gmin$ and $\gmax$ are defined for $\cB$ by \eqref{eq:gmax} with the given~$\kappa$.
We treat $\cF_\kappa$ as a set of genus--basket pairs, or equivalently as a set of rational functions they determine, and we write $(g,\cB)\in\cF_\kappa$ or $P_{g,\cB}\in\cF_\kappa$ accordingly.
\end{defn}

It follows from everything so far that if $X$ is a Fano~\tf\ with genus~$g$ and basket~$\cB$ that
satisfies~\eqref{eq!kaw} for $\kappa>0$, then $P_X=P_{g,\cB}\in\cF_{\kappa}$.
The condition on $P_n$ holds for any such $P_X$: there cannot be a nonzero section $x\in  H^0(X,-K_X)$
with $x^n=0$ for some $n>1$.
\begin{prop}\label{prop!list}
$\#\Fss=39,\!550$ and $\#\FMF=52,\!646$.
\end{prop}

\begin{proof}
Lemma~\ref{lem!baskets} provides exactly~$8314$ possible baskets. Of
these, \eqref{eq:gmax} calculates $\gmax\ge-2$ in~$7492$ cases for~$\kappa=3$
and in~$7683$ cases for~$\kappa=4$.
Assembling pairs~$(g,\cB)$ with~$\cB$ one of
these baskets and $\gmin\le g\le\gmax$ bounded by Proposition~\ref{cor!g} gives~$39,\!558$ cases for $\kappa=3$
and~$52,\!654$ cases for $\kappa=4$.

For each genus--basket pair~$(g,\cB)$, expand the rational function $P_{g,\cB}$ as a power series.
In eight cases, this power series starts
\[
1+t+t^2+\cdots+t^n + O(t^{n+2}),
\]
with either~$n=2$ or~$n=4$, where the~$t^{n+1}$ term has coefficient zero. 
We discard these eight series at this stage,
to leave~$39,\!550$ series in~$\Fss$ and~$52,\!646$ series in~$\FMF$.
\end{proof}

Computer code to enumerate the Hilbert series of Proposition~\ref{prop!list}, either in the Magma system~\cite{magma} or in the Go-language~\cite{Go} is
available at~\cite{grdbweb,fanocode}.

\section{The geography of Fano 3-folds}\label{sec!geography}
For each rational function~$P\in\FMF$, the \grdb\ gives a collection of
weights $a_0,\ldots,a_n$ for which the product
\begin{equation}\label{eq!hilbnum}
P \cdot \prod_{i=0}^{n} \left(1 - t^{a_i}\right) = 1 - \sum_{d_i} t^{d_i} +
\sum_{e_i} t^{e_i} - \cdots \pm t^k.
\end{equation}
is a polynomial. This polynomial is called the~\emph{Hilbert numerator
of~$P_X$}, and clearly it depends on the choice of weights~$a_i$. The point is
that if there really is a Fano~\tf~$X$ embedded as
\begin{equation}\label{eq!XinP}
X \subset \PP^n(a_0,\dots,a_n)
\end{equation}
then the expression~\eqref{eq!hilbnum} is related to the equation degrees~$d_i$,
syzygy degrees~$e_i$ and adjunction number~$k$ of the defining equations;
see~\cite[3.6]{kinosaki}. To be suggestive and provocative (and usually helpful), \grdb\ presents each
genus--basket pair~$(g,\cB)$ in the form~\eqref{eq!XinP}, for weights chosen to
suit the corresponding series~$P_{g,\cB}$. One must be aware that there could be
many different apparently `good' choices of weights, and it is important to
understand how the \grdb\ weights are chosen, since, as we explain
in~\S\ref{sec!examples}, there are many traps to fall into when interpreting
them.

The process to assign weights to Hilbert series $P$ in the \grdb\ follows four steps:
identify the known classification of weighted complete intersections (\S\ref{sec!famous95}); 
apply simple numerics to identify the weights of necessary generators
for any graded ring with given $P$ (\S\ref{sec!lowcodim});
additional generators may be needed to polarise the singularities of $\cB$ (\S\ref{sec!polarising});
finally an inductive Ansatz based on projection gives the weights in high codimension (\S\ref{sec!unproj}).
Which steps are used most varies with genus, as we explain below,
but in any case the weights provided by any
step are consistent with those from all other steps.

\subsection{Hilbert series in low codimension}\label{sec!lowcodim}
An embedding $X\subset\PP(a_0,\dots,a_n)$ is {\em nondegenerate} if
there are no quasi-linear equations in the defining ideal of $X$, that is, no equations which contain a linear term
and thereby may be used to eliminate a variable.
Anticanonical embeddings \eqref{eq!acemb} of Fano \tfs\ are nondegenerate, as they are
determined by the choice of a minimal generating set.
\subsubsection{The famous $95$ and the result of Chen--Chen--Chen}\label{sec!famous95}
The famous~$95$ weighted hypersurfaces of Reid~\cite[(4.5)]{c3f},
Johnson--Koll\'ar~\cite{jk,BK}, and others realise the codimension~$1$ (top) row
of Table~\ref{fig!geography}. This classification of hypersurfaces is well
established: if a prime Fano~\tf\ is anticanonically embedded as a hypersurface
in weighted projective space, then it is in one of the~$95$ families. The
converse is false: a variety may have the Hilbert series of one of the~$95$
without being a hypersurface, such as a complete
intersection~$X_{2,4}\subset \PP(1^5,2)$ with non-quasi-linear quadric
(cf.\ other degenerations in \cite[Table~2]{k3db}).

In similar vein, Iano-Fletcher,~\cite[(16.7) Table~6]{fletcher} lists~$85$
families of Fano~\tfs\ in codimension~$2$, realising the codimension~$2$
(second) row of Table~\ref{fig!geography}. Chen--Chen--Chen~\cite{CCC} prove
that this list is complete in the following sense:
if~$X\subset\PP(a_0,\dots,a_5)$ is a codimension~$2$ complete intersection
prime Fano~\tf, then either it is in one of Iano-Fletcher's~$85$
families, or it is a degeneration of one of the famous~$95$.
Again the converse is false: for
example,~\cite[Table~3]{specialK3} lists degenerations of~$13$ of the~$85$
families that lie in codimension~$3$ (described as~K3 surfaces, but each extends
to a Fano~\tf\ with an additional variable of degree~$1$).

\subsubsection{The graded ring game}\label{ref:grgame}
Suppose~$X\subset\PP=\PP(1^{m_1},2^{m_2},\dots)$ is a variety in a nondegenerate projectively
normal embedding. Such~$X$ has a Hilbert
series~$P_X(t) = \sum_{i\ge0} n_it^i$, where~$n_i = h^0(X,\cO_X(i))$. In
general, knowing~$P_X(t)$ alone is not enough to determine the~$m_j$, but it can
provide estimates, even without information about the equations of~$X$.

The following `graded ring game' is standard.
Given a series~$P_0 = 1+n_1t+n_2t^2+\cdots$, we attempt to build some
$X$ as above, polarised by $\cO_X(1)$, with $P_X=P_0$.
By the assumptions, we must start with~$m_1=n_1$.
Write~$P_1 = (1-t)^{m_1}P_0 = 1 + n_2't^2 + \cdots$. If~$n_2'\ge0$, then
there are necessarily at least $n_2'$ variables in weight~$2$, so we set $m_2=n_2'$.
(Of course we may add more variables of weight~2, but each one would necessitate
an equation of weight~2, and at least for a complete intersection 
general in its deformation family that would be excluded by nondegeneracy.)

Continuing, write~$P_2=(1-t^2)^{m_2}P_1=1+n_3''t^3 + \cdots$. If
now~$n_3''\ge0$, then there are necessarily at least $n_3''$ variables in
weight~$3$, so again we set $m_3=n_3''$ and we consider~$P_3=(1-t^3)^{m_3}P_2$.
At some stage~$n_{r+1}^{(r)}<0$ and the game ends: we can no longer infer that there must be
additional generators in degree $r+1$.
Interpreting the Hilbert numerator according to \eqref{eq!hilbnum} we may
now attempt to
construct~$X\subset\PP(1^{m_1},2^{m_2},\ldots,r^{m_r})$ using suitably general polynomials
of degrees~$d_i$, and check whether $P_X=P_0$ and whichever properties --
irreducible, quasismooth, Fano, and so on -- that we require do in fact hold.

As we saw, at any stage there may be more generators in the degree we are considering,
but there cannot be fewer. Indeed, at the $i$th stage
\[
P_0 = 
\left(\tfrac{1}{(1-t)^{m_1}\cdots(1-t^i)^{m_i}}\right)\cdot\left(1 + n_{i+1}^{(i)}t^{i+1} + \text{ higher order terms}\right)
\]
where the first factor is 
the Hilbert series of $\PP(1^{m_1},\dots,i^{m_i})$ (with no relations),
so in degree $i+1$
\[
n_{i+1}^{(i)} = n_{i+1} - N_i
\]
where $N_i$ is the maximum possible span of lower degree terms; if $n_{i+1}^{(i)}>0$
then the ring needs at least that many additional generators in degree~$i+1$.

When $X$ is a complete intersection, this approach with a small twist
determines the unique weights and degrees (cf.\ \cite[\S18]{fletcher}).
To illustrate the twist, consider $C_{2,6}\subset\PP(1^3,3)$, a genus~$5$ hyperelliptic curve
polarised by $2A$, where $A$ is the hyperelliptic linear system $g^1_2$.
This polarised curve has
Hilbert series~$P=1+3t+5t^2+8t^3+\cdots$, so the first step is to
consider~$(1-t)^3P = 1-t^2+t^3-t^5$. The game is complete, but clearly
there is no variety defined by a single quadric in~$\PP^2$ with a linear syzygy.
Speculatively looking ahead for the next positive coefficient suggests
considering a variable of weight~$3$, which recovers the numerical
data of~$C$ as~$(1-t^3)(1-t)^3P = 1-t^2-t^6+t^8$. In practice, this phenomenon
is rare, and when it does arise the solution is as simple as this.

Even when $X\subset\PP$ is not a complete intersection, this graded ring game provides a lower
bound for the weights $m_i$ of a minimal generating set, and so
provides both a lower bound on the actual codimension of~$X$
and a necessary subset of the weights.

Although this seems to give rather a lot away, when
applied to the series~$P_{g,\cB}\in\Fss$ with low genus or small basket with
small singularity indices it works well.
For example, with
\[
g=-1\quad\text{and}\quad 
\cB = \left\{\tfrac{1}2(1,1,1),\tfrac{1}3(1,1,2),\tfrac{1}5(1,1,4),\tfrac{1}5(1,2,3),\tfrac{1}7(1,1,6)\right\}
\]
this identifies necessary weights $1,5,6,7,8,9,10$ and Hilbert numerator
$1-t^{16}-t^{17}-t^{18} - t^{19} - t^{20} + t^{25} + t^{26} + t^{27} + t^{28} + t^{29} - t^{45}$.
This numerical data provides enough information to write down a Fano \tf\ $X_{16\dots20}\subset\PP(1,5,6,7,8,9,10)$ as the vanishing of the five maximal Pfaffians
(of degrees $16,\hdots,20$)
of a skew $5\times5$ matrix of forms (cf. \cite{BE,kinosaki,wgrass,formats}).

\subsubsection{Polarising baskets}\label{sec!polarising}
We also add weights to realise the singularity basket correctly. For example, $g=2$,
$\cB = \left\{ \tfrac{1}2(1,1,1), \tfrac{1}3(1,1,2) \right\}$ determines a
series~$P(t)$ that satisfies
\[
(1-t)^4(1-t^2)^2 P(t) = 1 - 3 t^4 + 3 t^6 - 2 t^7 + 3 t^9 - 3 t^{10} + \cdots.
\]
Thus the graded ring game only demands generators in degrees~$1,1,1,1,2,2$,
whereas there must be some ambient orbifold locus with stabiliser~$\bmu_3$ to accommodate the index~$3$
orbifold point. That could be achieved with a generator in degree~$3$, or a
combination of generators whose degrees have 3 as their greatest common divisor.
Here the simplest thing works:
\[
(1-t)^4(1-t^2)^2(1-t^3) P(t) = 1 - t^3 - 3 t^4 + 3 t^6 + t^7 - t^{10}
\]
suggests a variety~$X\subset\PP(1,1,1,1,2,2,3)$ in codimension~$3$ defined by
five Pfaffians of degrees~$3$,~$4$,~$4$,~$4$ and~$5$, and one can use this numerical 
hint to construct such a Fano~\tf.
(Notice that the equation
of degree~$5$ is {\em masked} in the Hilbert numerator by a syzygy of degree~$5$: the
Hilbert numerator is `really'~$1 - t^3 -3t^4 - t^5 + t^5+3t^6+t^7-t^{10}$.
Knowing that codimension~$3$ Gorenstein ideals have an odd number of generators
defined as Pfaffians is extra information that comes from the
Buchsbaum--Eisenbud theorem \cite{BE}.)

Other ways to introduce index~$3$ points, such as introducing weights~$6$ and~$9$,
may work but increase codimension. 
In practice, by making this check in descending order of singularity index, 
such questions are avoided, since additional high-degree variables may already polarise
lower-degree singularities.

By this stage, for each Hilbert series we have a collection of generator
weights that are enough to span low-degree pieces of the graded ring and to polarise the
singularities. This already determines the \grdb\ weights for
codimension~$\le3$, but is not always correct in high codimension. We fix this
next.

\subsection{Numerical unprojection Ansatz and weights}\label{sec!unproj}
Type~I Gorenstein unprojection~\cite{PR04,KM} is a technique that takes as input
a pair of Gorenstein schemes~$D\subset X$, with~$D\subset X$ of codimension~$1$,
and returns a new Gorenstein scheme~$Y$. In applications to projective geometry,
it often corresponds to a birational contraction of~$D$ to a point of~$Y$, and
that is how we wish to apply it.

\subsubsection{The Type I unprojection Ansatz}\label{sec!ansatz}
We describe a model case of Kustin--Miller unprojection following
Papadakis--Reid~\cite[2.4--2.9]{PR04}. Consider the following hypothetical
input-output process:

\medskip
\noindent{\textbf{Input:\ }
Let~$X\subset\PP(a_0,\dots,a_n)$ be a Fano~\tf\ with terminal singularities in
its anticanonical embedding with basket~$\cB_X$. Suppose~$D\subset X$ for some
coordinate plane~$D=\PP(a_i,a_j,a_k)$ with weights~$(a_i,a_j,a_k)=(1,a,b)$ in some order
and~$\gcd(a,b)=1$, and suppose moreover that~$X$ is quasismooth away from
finitely many nodes~$\Sigma\subset D$.

\medskip
\noindent{\textbf{Output:\ }
A quasismooth Fano~\tf~$Y\subset\PP(a_0,\dots,a_n,r)$ in its anticanonical
embedding, with~$r=a+b$, such that:
\begin{enumerate}
\item\label{I:1}
$Y$ contains the point~$P=(0:\cdots:0:1)$ and~$P\in Y$ is a terminal quotient
singularity~$\tfrac{1}r(1,a,b)$.
\item\label{I:3}
$Y$ has the same genus as~$X$. Moreover, if~$a+b+c>d$ for
every~$c,d\in\{a_0,\ldots,a_n\}\setminus\{a_i,a_j,a_k\}$, then none of the equations
of~$Y$ has nonzero constant or linear terms.
\item\label{I:2}
The Gorenstein projection from~$P\in Y$ is a birational map $Y\dashrightarrow X$
that factorises into birational morphisms as follows:
\[
\renewcommand{\arraycolsep}{.25em}
\begin{array}{ccccc}
&&Z&&\\
&\swarrow&&\searrow&\\
X&&&&Y
\end{array}
\]
Starting from $X$ with its unprojection divisor $D=\PP(1,a,b)\subset X$, 
the map $Z\rightarrow X$ is the small crepant~$D$-ample resolution of the
nodes of~$X$ (the contraction of finitely many $\cO(-1)\oplus\cO(-1)$ flopping curves, viewed
from~$Z$).
The midpoint $Z$ contains $\PP(1,a,b)$, the birational transform of $D$,
meeting these exceptional curves transversally,
and the morphism $Z\rightarrow Y$ contracts this $\PP(1,a,b)$ to~$P\in Y$ (which, viewed from~$Y$, is the Kawamata blowup of~$P\in Y$).
\item\label{I:4}
The new basket~$\cB_Y$ of~$Y$ is obtained by replacing the two
  orbifold points $\frac{1}{a}(1,b,-b)$ and $\frac{1}{b}(1,a,-a)$ on $D$ in $X$
  by the image point $\frac{1}{r}(1,a,b)$ of the contraction of $D$: that is,
\begin{equation}\label{eq!BXY}
\cB_X \cup \left\{\tfrac{1}{r}(1,a,b)\right\} =
\cB_Y \cup \left\{\tfrac{1}{a}(1,b,-b),\tfrac{1}{b}(1,a,-a)\right\}
\end{equation}
where $\tfrac{1}{a}(1,b,-b)$ is omitted if~$a=1$ and $\tfrac{1}{b}(1,a,-a)$ is omitted if~$b=1$.
\item\label{I:5}
The Hilbert series of~$Y$ is
\[
P_Y = P_X + \tfrac{t^{a+b}}{(1-t)(1-t^a)(1-t^b)(1-t^r)}.
\]
\end{enumerate}

This process, part \eqref{I:2} in particular, 
is expressed as standard moves of the Mori MMP.
In practice, it may often be described as a simple-minded geometric progression (or elimination
of one variable), as in the next example. 

\begin{example}[cf.\ {\cite[\S2]{kinosaki}}]
Consider $X=(x_0A-x_1B=0)\subset\PP^4$,
where $A,B$ are general forms of degree~3  in coordinates $x_0,\hdots,x_4$ of $\PP^4$.
Any such $X$ contains the plane $D=\PP^2=(x_0=x_1=0)$ and has $9$ nodes on~$D$.
Now, with slight abuse of notation for the coordinates and using the same $A,B$, construct
$Y=(yx_0=A, yx_1=B)\subset\PP(1^5,2)$ in coordinates $x_0,\hdots,x_4,y$.
The point $P=(0:\cdots:0:1)\in Y$ is a quotient singularity $\tfrac{1}2(1,1,1)$ -- the equations eliminate $x_0$, $x_1$ near $P$ -- with 9 smooth rational curves of degree~$\tfrac{1}2$ defined by $A=B=0$ passing through it.
The Gorenstein projection from $P\in Y$ is the elimination of $s$ from the
equations of $Y$, a birational map $Y\rightarrow X$ that factorises
as an extremal divisorial contraction $Z\rightarrow Y$ with centre $P\in Y$ followed by the contraction of the birational transform in $Z$ of the nine curves, which are now flopping curves orthogonal to~$-K_Z$.
The Hilbert series of $X$ and $Y$ are related by
\[
\frac{1-2t^3+t^6}{(1-t)^5(1-t^2)} = \frac{1-t^4}{(1-t)^5} + \frac{t^2}{(1-t)^3(1-t^2)}
\]
as in \eqref{I:5} (see \cite[\S2]{kinosaki} for more discussion).
\end{example}

In many cases this process is a theorem (see, e.g.,~\cite[3.2]{tj}). Indeed
the setup~$D\subset X$ satisfies the conditions for Kustin--Miller unprojection
\cite[2.4]{PR}, giving a new variable~$s$ of degree~$r=k_X-k_D=-1-(-1-a-b)=a+b$
and additional equations involving~$s$ of the form~$sf_i=g_i$, where~$f_i$ form
a basis of the ideal~$I_D$ in the coordinate ring of~$X$. One can see using a
free resolution of the coordinate ring~$\C[Y]$ over~$\C[\PP(a_0,\dots,r)]$
\cite{papadakiswith} that~$\cO(-K_Y)=\cO(1)$, and the numerics
of~\eqref{I:1},~\eqref{I:3},~\eqref{I:4}, and~\eqref{I:5} follow
(cf.~\cite[2.7--2.9]{PR}). Then given~$Y$, the Kawamata blowup of~$P\in Y$ is a
weak Fano, and it has an anticanonical model~$Z\rightarrow X$. More complicated
situations can arise~--~see~\eqref{eq!XYZ}, where~$Z\rightarrow X$ makes both a
crepant divisorial contraction and a disjoint flopping contraction~--~but since
the assumption of nodes here already establishes that contracting the flopping
curves result in a Fano, there can be no further contraction in the given
situation.

The key point here is that we do not use the setup above as a theorem to be applied, rather
we turn it around to act as an Ansatz, as follows.

\begin{ansatz}[Type I unprojection]\label{typeIansatz}
Consider a genus--basket pair~$(g,\cB_Y)\in\FMF$ that is not among the~$95+85$
cases assigned weights in codimension~$1$ or~$2$ by~\S\ref{sec!famous95}.
Suppose that~$(g,\cB_X)$ is another genus--basket pair (with the same genus~$g$) for some $X$  which
satisfies \eqref{eq!BXY} with $r=a+b$ and $a, b$ coprime to~$r$. If the
weights~$(a_0,\ldots,a_n)$ of~$X$ listed in \grdb\ contain~$(1,a,b)$ as a
sublist, then we insist that the weights of~$Y$ in \grdb\
are~$(a_0,\ldots,a_n,r)$.
\end{ansatz}

\subsubsection{Type I consistency}
Ansatz~\ref{typeIansatz} works as an inductive procedure from low to high codimension
and is simple to
arrange in any particular case. The main point is that this
operation is well defined over the whole Fano~\tf\ database and is consistent
with the necessary weights identified in \S\ref{sec!lowcodim}; the proof is 
a computer check on the collection of all weights.

\begin{lemma}\label{lem:type1consistency}
Suppose $(g,\cB_Y)\in\Fss$ admits the Type~I unprojection
relation of Ansatz~\ref{typeIansatz} for two different pairs~$(g,\cB_{X_1})$
and~$(g,\cB_{X_2})$ with respect to (possibly different) sums $r_1=a_1+b_1$ and $r_2=a_2+b_2$.
Then the weights for~$(g,\cB_Y)$ that Ansatz~\ref{typeIansatz} prescribes
are independent of which~$X_i$ pair is used, and they
include the minimal weights determined in~\S\ref{ref:grgame}.
\end{lemma}

To give some idea of the potency of this result, only~$1087$ of the~$39,\!370$
genus--basket pairs not among the~$95+85$ cases in codimension~$\le2$ fail to 
satisfy the Type~I projection numerics for~$(g,\cB_Y)$ in Ansatz~\ref{typeIansatz}. 
But to be clear: the claim is not that for each of the 38283 $(g,\cB_Y)$ pairs
with Type~I projection numerics may
find a particular~$D\subset X$ that satisfies the conditions specified as input
to a Type~I unprojection above: it is only that the numerics of the
weights are consistent with its existence. There is no promise that we will
be able to make unprojections in accordance with Ansatz~\ref{typeIansatz} in every
case to realise most of the Hilbert series
(compare~\S\ref{sec!unprojection} for an attempt to tackle this stronger claim).

\begin{remark}\label{rk:higherunproj}
Type I Gorenstein projection eliminates a single variable while preserving the
anticanonical ring. There are many other such operations. 
The bigger story includes Fano's historical
methods based on linear projection (see \cite[\S9]{kinosaki} for a general
discussion of higher Gorenstein projections).
The class of Type~II$_n$ 
projections for~$n\ge1$ (cf.\ \cite{papadakis08,Taylor}) is one such higher Gorenstein projection.
These may be used to describe weights for genus--basket pairs via the
relation~\eqref{eq!BXY} when some multiple $(n+1)c$ (for $n\ge1$ minimal)
of one of the polarising
weights~$c\in\{1,a,b\}$ lies among the weights of~$X$ but $c$ itself does not.
In this case, the unprojection adjoins~$n+1$ variables of weights~$r,r+c,\dots,r+nc$.

For example, $\tfrac{1}7(1,2,5)\in Y\subset\PP(2,3,4,5,6,7,7,8,9)$ satisfies these
numerics for Type~II$_2$.
That suggests a construction of $Y$ by unprojection of a suitable divisor $D\subset X_{12,14}\subset\PP(2,3,4,5,6,7)$; this construction is carried out in \cite{Taylor}.

Making a similar Ansatz for Type~II$_n$ unprojection results in an analogous consistency
result to Lemma~\ref{lem:type1consistency} (cf.\ \cite[3.4]{k3db}). 
This does not change the results of the Type~I constructions, but does help
in very low genus: when $g=-2$, there are no Type~I projections at all (there is no 1 among the weights).
More generally, $735$ of
the $52,\!646$ genus--basket pairs do not admit a numerical Type~I projection,
but do admit a numerical Type~II$_1$ projection (and, in fact, all lie
in~$\Fss$); a further~$159$ admit a Type~II$_2$ projection but no Type~I or II$_1$; then a further~$68$ have
Type~II$_3$ but no Type~I, II$_1$ or II$_2$, 24 have Type~II$_4$ but no Types~I--II$_3$, and so on with diminishing returns.
\end{remark}

\begin{remark}
The birational projections we consider in \S\ref{sec!ansatz} endow the set $\Fss$ with a directed graph structure:
for $(g,\cB_X),(g,\cB_Y)\in\Fss$ record a directed edge $(g,\cB_Y)\rightarrow (g,\cB_X)$
whenever there is a  Type~I  unprojection $D\subset X\dashrightarrow Y$ (genus is preserved by such unprojections).
Fixing any $g\ge1$ gives a connected component of this projection graph.
This idea provides an approach to classification: construct those Fano \tfs\ that have no projection,
and then attempt to make all possible unprojections, working serially up through the projection graph.
The case $g=0$ gives a connected component only if we include Type~II$_1$ projections as in Remark~\ref{rk:higherunproj} (otherwise there are 3 isolated cases, each of which is a hypersurface),
and $g=-1$ gives another connected component only if we include such higher unprojections
up to Type~II$_7$. When $g=-2$, if we permit any Type~II or Type~IV unprojection (see \cite{typeIV}), then the 264 cases split into 5 connected components. This includes two cases, numbers 1 and~6 of \grdb\ in \cite{grdbweb}, that have no Gorenstein projection or unprojection at all.
\end{remark}

\subsection{Numerical corollaries}\label{sec!corollaries}
The crude classification by Theorem~\ref{thm!main} already contains enough information
to provide approximations to various strong and sharp theorems by elementary and
easy means. For example, recall Prokhorov's sharp bound on the degree.

\begin{thm}[Prokhorov~\cite{prokhorov07}]
For a prime Fano~\tf\ $X$ with Gorenstein index $>1$,
the anticanonical volume $-K_X^3 \le \tfrac{5^3}2$, and this bound is realised only
by~$X=\PP(1^3,2)$ for which $-K_X = 5A$ and $A^3=\tfrac{1}2$.
\end{thm}

Theorem~\ref{thm!main} recovers the weaker bound~$-K_X^3\le 72$ at once. Applying
\cite{karzhemanov66,karzhemanov64} sharpens this to~$-K_X^3\le66\tfrac{1}2$.
The next case to exclude is~$(g,\cB)=(34,\{\tfrac{1}2(1,1,1)\})$, which we see
(\S\ref{sec!g34}) is populated by the blowup $\Bl_P\PP(1,1,3,5)$ at a smooth
point~$P$, a Fano \tf\ with the non-terminal point $\tfrac{1}5(1,1,3)$.
The numerical methods alone cannot say more: instead,
Prokhorov's stronger geometric Sarkisov methods show this case cannot
be realised by a prime Fano~\tf.

Another example is the following result of~\cite{CCC},
proving~\cite[Conjecture~18.19(2)]{fletcher}, which bounds the codimension of
Fano complete intersections.

\begin{cor}[\!{\cite[Theorem~1]{CCC}}]
A prime Fano~\tf\ $X$ with anticanonical model a weighted complete intersection
$X\subset w\PP^n$ has codimension $n-3\le 3$.
\end{cor}

The proof in \cite{CCC} is difficult and subtle, but with Theorem~\ref{thm!main} and Corollary~\ref{cor!ss}
it follows by inspection of the numerators of all
Hilbert series in the database. Indeed, for a complete intersection, the process
of determining generic weights and numerator (\S\ref{sec!lowcodim}) is well defined.
Conversely, given weights, the numerator determines a minimal set of degrees for
equations. In most cases, this immediately rules out a complete intersection, as
there are too many equations -- and adding additional weights does not alter
that. Of the remaining, there are cases of apparent complete intersections, but
by equations whose degrees are too small to accommodate the high-degree
variables: therefore to be terminal there must be further equations, and again
complete intersection is ruled out.

\section{Populating the Fano 3-fold database}\label{sec!known}
We locate some established classes of
Fano~\tfs\ within the Fano~\tf\ database~$\Fss\subset\FMF$.
We consider Fano \tfs\ that are:
\S\ref{sec!smooth}~nonsingular, 
\S\ref{sec!gorenstein}~Gorenstein, 
\S\ref{sec!index2}~Gorenstein index~2, 
\S\ref{sec!higherindex}~higher index, 
\S\ref{sec!toric}~toric, 
\S\ref{sec!formats}~in Gorenstein formats, 
\S\ref{sec!unprojection}~unprojections, and finally in
\S\ref{sec!fanosearch} those arising in the fanosearch program as qG-smoothings of toric varieties. 
This gives a much more detailed understanding of Table~\ref{fig!geography},
and pinpoints places where the next questions lie.
We give cases of $P\in\Fss$ which have
no matching prime Fano~\tf\ (see \S\ref{sec!smooth}, \ref{sec!index2}) and others
where there are multiple deformation families (cf.\ Table~\ref{fig!takagi}), but this
analysis is far from complete.
More generally, we know only a single element of~$\Fss$
for which it is proven that there is no matching Fano~\tf\ with canonical
singularities (see~\S\ref{sec!gorenstein}), and we give examples of such Fano \tfs\ whose
Hilbert series do not appear in~$\Fss$ (see \S\ref{sec!toric}).

\subsection{Nonsingular Fano 3-folds}\label{sec!smooth}
The classification of 105 families of nonsingular Fano~\tfs\
\cite{I2,MM,bkm24} is a cornerstone of higher-dimensional birational geometry.
It is listed in~\cite[Table~12.2]{IP} and online at~\cite{fanography}.
Each of the 105 families lies in its anticanonical embedding along the leading diagonal of Table~\ref{fig!geography}:
Table~\ref{fig!MM} and its explanation shows how. 
We should expect similarly subtle behaviour
in hundreds of places throughout Table~\ref{fig!geography}, so it is useful to
absorb this familiar case first.

\begin{table}[ht]
\[
\small
\begin{array}{c|cc|cccc|cccc|c|c|c}
\toprule
g&\multicolumn{2}{c|}{\MM{1}{n}}&\multicolumn{4}{c|}{\MM{2}{n}}&\multicolumn{4}{c|}{\MM{3}{n}}&\MM{4}{n}&\MM{\rho}{n}&\#\cT\\
\cmidrule(lr){1-14}
2&1&&&&&&&&&&&&0\\
3&2&&&&1_{11}&&&&&&&&1\\
4&3&&&&&2&&&&&&\MM{10}{1}&7\\
5&4&11&&&3_{12}&&&&&&&&23\\
6&5&&4^{*}&&&&&&&&&&54\\
7&6&&&&5_{13}&6&&&&1&&\MM{9}{1}&135\\
8&7&&&7^{\dagger}&&8&&&&2&&&207\\
9&8&12&&&10_{14}&9&&&&&&&314\\
\cmidrule(lr){1-14}
10&9&&&&11_{13}&&3^{*}&&4_{18}&&&\MM{8}{1}&373\\
11&&&12^*&13^\dagger&14_{15}&&5^*&&&&&&416\\
12&10&&15^*&&16_{14}&&&&6_{25,33}&&&&413\\
13&&13&17^{*\dagger}&&&18&7^*_{32}, 8^*_{24}&&&&1^*&\MM{7}{1}&413\\
14&&&19^*_{14}&&20_{15}&&&&9_{36}, 10_{29}&&13^*&&348\\
15&&&&21^\dagger&&&12^*_{27,33}&11^{*\dagger}_{25}&&&2_{31}&\MM{5}{1}&344\\
16&&&22^*_{15}&23^\dagger&&24&&&13_{32}&&3^{*}_{17,28}&\MM{6}{1}&274\\
17&&14&25^*&&&&15^{*}_{29,31}&14^{\dagger}_{36}&&&4_{18/9}^{30}, 5^{31}_{21/8}&&234\\
\cmidrule(lr){1-14}
18&&&&26^{\dagger}_{15}&&&&16^{\dagger}_{27,32}&&&6^{*}_{25}&&179\\
19&&&&&&&17^*&&18_{29,30/3}&&7_{24,28}&\MM{5}{2/3}&151\\
20&&&27^*&&&&21^*&19^\dagger&20_{31/2}&&8^{*}_{31}&&117\\
21&&15&28^*&29^\dagger&&&22^*_{36}&&&&9_{25/6/8,30}&&87\\
22&&&&&&&24^{*}_{32}&23^{\dagger}_{30/1}&&&10^{*}_{28}&&66\\
23&&&&&&&&&25_{33}&&11_{28,31}&&40\\
24&&&30^*&31^\dagger&&&&26^{*\dagger}&&&12_{30}&&42\\
25&&&&&&32&28^*&&&27&&&27\\
\cmidrule(lr){1-14}
26&&&&&&&&29^\dagger,30^{\dagger}_{33}&&&&&18\\
27&&&&&&&&&&31&&&8\\
28&&16&33^*&&&34&&&&&&&13\\
29&&&35^*&&&&&&&&&&9\\
30&&&&&&&&&&&&&4\\
31&&&&&&&&&&&&&2\\
32&&&&&&36&&&&&&&2\\
33&&17&&&&&&&&&&&5\\
\cmidrule(lr){1-14}
&&&{}^*_{\!17}\PP^3&{}^\dagger_{\!16} X_2&&&{}^*_{\!34}\PP^1\!\times\!\PP^2&{}^\dagger_{\!35}\widehat{\PP^3}&&&{}^*_{\!27}(\PP^1)^3&\\
\bottomrule
\end{array}
\]
\vspace{0.5em}
\caption{The~$105$ families of nonsingular Fano~\tfs, listed as~$\MM{\rho}{n}$ for
the~$n$th variety of Picard rank~$\rho$, with a row for each
genus~$g=2,\dots,33$.}
\label{fig!MM}
\end{table}

Table~\ref{fig!MM} uses the numbering convention of~\cite{fanography} (which adapts~\cite{IP,MM})
to list the 105 families: $\MM{\rho}{n}$ denotes
the~$n$th family of \tfs\ with Picard rank~$\rho$.
The table is indexed in columns by Picard rank $\rho=1,\dots,10$ across the top (the sparse cases $\rho=5$--10
are compressed into a single column).
The entries of the table are the family numbers $n$ (in full-size font) for the $n$th
family with that~$\rho$;
these entries often have sub/superscript decorations that we will come to.
Each row is indexed by genus $g_X=2,\dots,33$ down the side, so that
each row lists all families of a given genus~$g$.
Since for a nonsingular Fano~\tf~$X$ the genus~$g_X$ determines~${-}K_X^3$ and~$P_X$,
one may imagine this table lying sideways along the leading
diagonal of Table~\ref{fig!geography}. For example, the nonsingular Hilbert series
for $g=13$ is realised by the 7 families $\MM{1}{13}$, $\MM{2}{17}$, $\MM{2}{18}$, $\MM{3}{7}$, $\MM{3}{8}$, $\MM{4}{1}$ and $\MM{7}{1}$,
and in Table~\ref{fig!geography} 
these all lie at the single position genus 13 and codimension 11 on the leading diagonal. 

The decorations and positioning within each column indicate additional information.
Fano manifolds~$X$ with~$\rho_X=1$, the first column, are divided with those
of Fano index~$f_X=1$ on the left-hand side
of column $\MM{1}{n}$, and~$f_X\ge2$ on the right-hand side.  
When~$\rho\ge2$, many of these Fano~\tfs\ are constructed by extremal extractions
from other nonsingular Fano~\tfs. The decorations on entries of the table
include some of these extremal
divisorial contractions by writing entries as~$\MM{\rho}{n_m}$, abbreviated to~$n_m$ in
the table, to indicate a map from members of family~$\MM{\rho}{n}$ to members
of~$\MM{(\rho-1)}{m}$. Some codomains are very common, and we indicate these by the
following special notation:

\begin{itemize}
\item
$\MM{2}{n^{*}}\equiv\MM{2}{n_{17}}$ means map to $\MM{1}{17}=\PP^3$
\item
$\MM{2}{n^{\dagger}}\equiv\MM{2}{n_{16}}$ means map to $\MM{1}{16}=X_2\subset\PP^4$
\item
$\MM{3}{n^*}\equiv\MM{3}{n_{34}}$ means map to $\MM{2}{34}=\PP^1\times\PP^2$
\item
$\MM{3}{n^\dagger}\equiv\MM{3}{n_{35}}$ means map to $\MM{2}{35}=\Bl_P\PP^3$
\item
$\MM{4}{n^*}\equiv\MM{4}{n_{27}}$ means map to $\MM{3}{27}=\PP^1\times\PP^1\times\PP^1$
\end{itemize}
Columns~$\rho=2$,~$3$ are arranged to give a brief indication of the extremal contraction data,
with varieties which admit maps to the common target listed on the left-hand side of each
column, with those which only admit  other morphisms to the right of centre, and finally
those with no birational contractions to another nonsingular Fano~\tf\ down the
right-hand side; most of the latter are products or double covers, and their extremal
rays are Mori fibrations.
For example, the diagram
\[
\begin{array}{rcccl}
&&\MM{3}{15}\\
&\swarrow&\downarrow&\searrow&\\
\PP^1\times\PP^2 = \MM{2}{34} && \MM{2}{29} && \MM{2}{31} \\
&&&\searrow&\ \ \downarrow\\
&&&&\MM{1}{16} = (X_2\subset\PP^4)
\end{array}
\]
indicates sequences of extremal contractions from $\MM{3}{15}$ that one can read from Table~\ref{fig!MM} (see \cite{fanography} for full details).
Finally, the right-hand column~$\#\cT$ lists the number of toric
models with Gorenstein canonical, but not necessarily $\Q$-factorial, singularities by genus (see~\S\ref{sec!toric}).

Generalising from smooth to Gorenstein terminal Fano~\tfs\
does not increase the number of families: any
Gorenstein terminal Fano~\tf\ may be smoothed \cite{namikawa} so it appears in
Table~\ref{fig!MM}. Moreover the Picard rank does not
change on smoothing (see \cite{JR}, which uses this to determine smoothing
families). In contrast, there are Fano~\tfs\ with canonical Gorenstein
singularities that do realise other families in the leading diagonal of
Table~\ref{fig!geography}, as we discuss next.

\subsection{Gorenstein Fano 3-folds}\label{sec!gorenstein}
When~$X$ is a Fano~\tf\ with canonical singularities that has~${-}K_X$ Cartier,
then its singularities are Gorenstein and its basket is empty. Therefore its
Hilbert series~$P_X$ behaves as though~$X$ has no singularities at all, and so
again lies on the leading diagonal of Table~\ref{fig!geography}.

There is no classification of Fano~\tfs\ with Gorenstein canonical singularities,
but there are results in many special cases: \cite{PCS} settles the hyperelliptic and trigonal cases,
\cite{JR06} for the case $|{-}K_X|$ has basepoints,
and \cite{bkm25} for the terminal case with $g\in\{6,7,8,9,10,12\}$.
There are also
precise results in certain extreme cases, which we describe now in relation to the
\grdb\ geography. Recall that in the bottom-right block of
Table~\ref{fig!geography} the number of Hilbert series in~$\Fss$ by genus and
codimension is:
\begin{equation}\label{eq!corner}
{\small
\begin{array}{cc|ccccc}
\multicolumn{2}{c}{}&\multicolumn{5}{c}{\textrm{Genus}}\\
&&33&34&35&36&37\\
\cline{2-7}
\multirow{5}*{\rotatebox{90}{\textrm{Codim}}\,}
&31&1\\
&32&1&1\\
&33&&1&1\\
&34&&&&1\\
&35&&&&&1
\end{array}}
\end{equation}
Prokhorov~\cite[1.5]{prokhorov05} proves that the largest possible genus for a
Fano~\tf\ with~${-}K_X$ Cartier is~$g=37$ with degree~${-}K_X^3=2g-2=72$
(strengthening Cheltsov's result~$-K_X^3\le184$~\cite{cheltsov184}). Moreover
Prokhorov proves that~$g=37$ is realised only by~$\PP(1,1,4,6)$
and~$\PP(1,1,1,3)$. That completely clears up the final column.

Cheltsov and Karzhemanov extend this to Gorenstein \tfs\ in genus~$g\ge33$.

\begin{thm}[\!\!\cite{karzhemanov66}]\label{thm!karzh66}
Let~$Z$ be a Fano~\tf\ with Gorenstein canonical singularities
and genus $34\le g\le36$.
\begin{enumerate}
\item\label{bl:1}
If~$g=36$ then~$Z$ is the image of $\widetilde{Z}=\Bl_p\PP(1,1,4,6)$
after contracting the flopping line $\widetilde{\Gamma}\subset\widetilde{Z}$,
where~$p\in\Gamma=\PP(4,6)$ is an index~$2$ point, $\Bl_p$ is the discrepancy~1
divisorial extraction at~$p$, and
$\widetilde{\Gamma}$ is the birational transform of $\Gamma$ in $\widetilde{Z}$.
\item\label{X66ii}
The case~$g=35$ is not possible.
\item
If~$g=34$ then~$Z$ is the anticanonical image of the projectivised bundle
$U=\Proj_{\PP^1} \cO\oplus\cO(2)\oplus\cO(5)$.
\end{enumerate}
\end{thm}
The point $p$ in \eqref{bl:1} is not a toric $0$-stratum, but for a local toric model of $\Bl_p$
blow up the toric cone $\left< (0,1,0), (1,0,1), (1,0,-1)\right>$ 
at $(1,1,0)$ (rather than the usual discrepancy~2 blowup at $(2,1,0)$).

\subsubsection{Genus~34}\label{sec!g34}
We consider Karzhemanov's degree~$66$ example~$Z$ in more detail.
The weighted projective space $X=\PP(1^2,3,5)$ is a Fano~\tf\ with
a terminal singularity~$\tfrac{1}3(1,1,2)$, a canonical
singularity~$\tfrac{1}5(1,1,3)$, genus~$34$ and degree~$-K^3=66\tfrac{2}3$. Its
Hilbert series lies in~$\FMF$ but not~$\Fss$, just outside
the region of Table~\ref{fig!geography}. It has anticanonical embedding
$X\subset\PP(1^{36},2,3)$
in the blank position~$(34,34)$ in~\eqref{eq!corner}. It fits into a diagram of
toric varieties and maps:
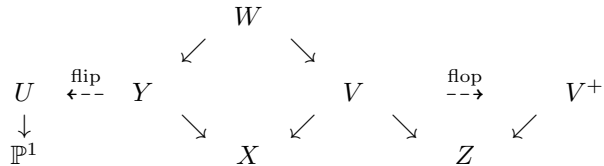
\[
\begin{array}{ccccccccccc}
&&&&W\\
&&&\swarrow&&\searrow\\
U&\buildrel{\text{flip}}\over{\dashleftarrow}&Y&&&&V&&\buildrel{\text{flop}}\over{\dashrightarrow}&&V^+\\
\downarrow&&&\searrow&&\swarrow&&\searrow&&\swarrow\\
\PP^1&&&&X&&&&Z
\end{array}
\]
where $Y\rightarrow X$ is the Kawamata~$\tfrac{1}3(1,1,2)$ blowup of~$P_3\in X$
(the index~$3$ point), $Y\dashrightarrow U$ is the~$(5,2,-1,-1)$ (canonical)
flip to Karzhemanov's bundle~$U\rightarrow\PP^1$, $V\rightarrow X$ is
the~$\tfrac{2}3$\nobreakdash-discrepancy $\tfrac{1}3(2,2,1)$ blowup
of~$P_3\in X$ and~$V\dashrightarrow V^+$ is the~$(5,1,-3,-3)$ (canonical) flop with base~$Z$,
while finally~$W$ is the resolution of~$P_3\in X$, which itself admits two flops to
divisorial extractions of~$U$ and~$V^+$ respectively. In this picture
\begin{equation}\label{eq!XYZ}
\left( X\subset\PP(1^{36},2,3) \right) \dashrightarrow
\left( Y\subset\PP(1^{36},2) \right) \dashrightarrow
\left( Z\subset\PP^{35} \right)
\end{equation}
is a sequence of projections between Fano~\tfs\ of
degrees~$66\tfrac{2}3$,~$66\tfrac{1}2$ and~$66$,
where~$\rho_X=\rho_Z=1$ and~$\rho_Y=2$, and~$X$ and~$Y$
are~$\Q$\nobreakdash-factorial while~$Z$ is not.

Loosely speaking, we see how knowing Karzhemanov's example provides other
Fano~\tfs\ by birational contractions; from the other end,
knowing~$\PP(1^2,3,5)$ provides other Fano~\tfs\ by birational blow ups. We
discuss this further in~\S\ref{sec!cascades}.

\subsubsection{Genus~33}\label{sec!g33}
Karzhemanov~\cite{karzhemanov64} also classifies the case of degree~$64$: there is $\PP^3$
and six other cases, all with strictly canonical singularities. The
\grdb\ matches five toric cases, and we illustrate with a beautiful
non-$\Q$\nobreakdash-factorial example
\[
\Phi_{{-}K_X}\colon
X = \TV_{\left(\begin{smallmatrix}5\\6\end{smallmatrix}\right)}\begin{pmatrix}1&1&5&3&0\\0&1&6&4&1\end{pmatrix}
\hookrightarrow\PP^{34}.
\]
The notation~$\TV_\chi M_{2\times 5}$ denotes a toric
variety expressed as the GIT quotient $\C^5 /\!\!/_\chi (\C^*)^2$ where
the rows of $M$ are the weights of the $(\C^*)^2$ action (see~\cite[\S{}A.5]{BCZ}). The map $\Phi_D$ is the natural map to $\Proj\oplus_{m\ge0}H^0(mD)$; in this case $D=-K_X$ has bidegree $(10,12)$ and, in coordinates $x,y,z,u,v$ of $\C^5$, $\Phi_{-K_X}$ is
given by the 35 monomials $z^2,xu^3,y^{10}v^2,\hdots$ of bidegree $(10,12)$.
This is the base of a~$(6,1,-2,-5)$ (canonical) flop
\[
\begin{array}{ccccccccc}
&&U&&\buildrel{\text{flop}}\over{\dashrightarrow}&&U^+\\
&\swarrow&&\searrow&&\swarrow&&\searrow\\
\PP(1^2,4,6)&&&&X&&&&\PP(1^2,3,5)
\end{array}
\]
where the map~$U\rightarrow\PP(1^2,4,6)$ is the blowup of a smooth point and
$U^+\rightarrow\PP(1^2,3,5)$ is a weighted~$\tfrac{1}3(1,2,4)$-blowup of the
index~$3$ point; that is,
\[
\text{blowup } (-1,-4,-6)\text{ in the cone }\left<(1,0,0),(0,0,1),(-1,-3,-5)\right>.
\]
(The numerics come from the observation that $(6,1,0,-2,-5)$, $(-1,0,1,1,1)$ and $(4,1,2,0,-3)$ lie in the span
of the rows of the defining matrix of~$X$.)

\subsection{Index~2 singularities}\label{sec!index2}
The list~$\FMF$ contains~$360$ pairs~$(g,\cB)$,
where~$\cB=\{N\times\tfrac{1}2(1,1,1)\}$ with~$N\ge1$ of which~$272$ lie
in~$\Fss$. Restricting to~$g\ge2$, these numbers reduce to~$325$
and~$238$ respectively.
Sano~\cite{sano95,sano96} and Campana--Flenner~\cite{CF}
classify such terminal Fano~\tfs~$X$
under the assumption Fano index $f_X\ge1$.
For non-empty baskets~$\{N\times\tfrac{1}2(1,1,1)\}$ and $f_X>1$,
these are \cite{sano96} (with their anticanonical embeddings mapping to the right) 
\[
\begin{array}{rclcl}
\PP(1^2,2^2,3) \supset X_6&\hookrightarrow&\PP(1^8,2^3)&\qquad&\text{($N = 3$, $g=6$, codim 7)}\\
\PP(1^3,2^2) \supset X_4&\hookrightarrow&\PP(1^{16},2^2)&&\text{($N = 2$, $g=14$, codim 14)}\\
\PP(1^4,2) \supset X_3&\hookrightarrow&\PP(1^{23},2)&&\text{($N = 1$, $g=21$, codim 20)}\\
\PP(1^3,2)&\hookrightarrow&\PP(1^{34},2)&&\text{($N = 1$, $g=32$, codim 31)}
\end{array}
\]
For the case~$f_X=1$, \cite[1.1]{sano95} presents a classification as follows: there are
14 families of nonsingular Fano~\tfs, particular members $Y$ of which admit a~$\bmu_2$ action with quotient $X=Y/\bmu_2$ having index~2 singularities and $f_X=1$.
The theorem constructs representative cases of each type, again all in rather high anticanonical codimension, but does not enumerate all possible $Y$ and $\bmu_2$ actions.

\subsubsection{Prime cases}
The remaining cases for index~$2$ baskets satisfy~$i_X=2$ and~$f_X=1/2$.
Takagi~\cite{takagi} classifies prime Fano~\tfs\ with such genus--basket pairs (imposing also~$g\ge2$)
into~$35$ families matching 23 elements of~$\Fss$.
They are presented in Tables~1--5 of~\cite{takagi}, with the individual families numbered~1.1, 1.2,~\dots,~5.5.
These are listed in Table~\ref{fig!takagi}, ranging from
Family~$3.1$,~$X_5\subset\PP(1^4,2)$, to
Family 1.14, $X\subset\PP(1^{10},2^2)$ in codimension~$8$,
with Type~I projections going up the columns.
(The remarkable division into 2 cases in codimension~4 is the starting point for \cite{tj}.)

\begin{table}[ht]
\[
\small
\setlength{\arraycolsep}{2mm}
\begin{array}{cc|c|c|c|c|c|c|c}
\multicolumn{2}{c}{}&\multicolumn{7}{c}{\text{Genus}}\\
&&2&3&4&5&6&7&8\\
\cline{2-9}
\multirow{8}*{\rotatebox{90}{\text{Codimension}}}&1&3.1&&&&&&\\
&2&3.2&5.2&&&&&\\
&3&2.1&5.3&4.3&&&&\\
&4&2.2,\  3.3&4.1,\ 5.4&1.1,\ 4.4&1.4&&&\\
&5&2.3,\  3.4,\  5.1&4.2,\  5.5&1.2,\  1.3,\  4.5&1.5,\ 1.6&1.9,\  1.10&&\\
&6&2.4&&4.6&1.7,\  1.8&1.11&1.12\\
&7&&&4.7&&&&1.13,\  4.8\\
&8&&&&&&&1.14
\end{array}
\]
\vspace{0.5em}
\caption{Families of Tables~$n.m$ of~\cite{takagi} as they appear in
Table~\ref{fig!geography},
arranged by genus~$g$ and codimension~$c$. The generic member
of each family
is embedded as~$X\subset\PP(1^{g+2},2^N)$ with basket~$\cB=\{N\times\tfrac{1}2(1,1,1)\}$
where~$N=c-g+2$.}
\label{fig!takagi}
\end{table}

The comparison with the geography in Table~\ref{fig!geography} is striking.
The~$272$ pairs~$(g,\cB)\in\Fss$ are spread over most of the table, away from the
top diagonal line of Gorenstein pairs~$(g,\emptyset)$,
and Takagi's result shows that most are not realised by prime Fano~\tfs.
Many of the remaining~$272-23=249$ pairs
are realised by more general Fano~\tfs\ (cf.\ \S\S\ref{sec:nonprime2} and \ref{sec!toric}),
though the complete classification of Fano \tfs\ of Gorenstein index~$2$ remains unknown.

\subsubsection{Non-prime cases}\label{sec:nonprime2}
The general hypersurface $X_5\subset\PP(1,1,1,2,3)$ has $-K_X=\cO(3)$
and $g=12$, with $\cB=\{\tfrac{1}2(1,1,1)\}$. 
Its actual singularities are the expected $\tfrac{1}2(1,1,1)$ together with an additional
canonical singularity $\tfrac{1}3(1,1,1)$ which has a crepant resolution and so does not 
contribute to the basket. The anticanonical embedding $X\subset\PP(1^{14},2)$
lies in codimension~11, and so $X$ occupies
 position of Table~\ref{fig!geography} that Takagi's classification does not fill.
 There are many other similar examples (cf.\ \cite{BK}).

Venkatachalam \cite{thamarai} classifies 251 families of terminal $\Q$-factorial Fano 3-folds $X$ of Picard rank $\rho_X>1$ that lie in rank~2 toric variety $\F$. (More precisely, the hypothesis is that the pair $(\F,X)$ is a Fano terminal pair, with $\F$ simplicial and $X$ not a cone.)
This provides other families of non-prime Fano \tfs\ with index~2 singularities.
For example, \#61 of \cite[\S2.2 Table~1]{thamarai2} (= \#29394 in \cite{grdbweb})
\[
X_{\left(\begin{smallmatrix}2\\4\end{smallmatrix}\right)} = \TV_{\left(\begin{smallmatrix}1\\2\end{smallmatrix}\right)}\begin{pmatrix}1&1&1&0&0&0\\0&1&1&1&1&2\end{pmatrix}
\quad\textrm{with $g=6$ and $\cB=\{2\times\tfrac{1}2(1,1,1)\}$}
\]
has anticanonical embedding into $\PP(1^8,2^2)$ in codimension~$6$.
(Recall the notation for toric varieties from~\S\ref{sec!g33}; $X_D$ denotes a
member of the linear system $|D|$, i.e., $f=0$ for a polynomial $f$ of the given bidegree $D$.)
This gives a different family having the same $g,\cB$ and codimension as Family~1.11 of \cite{takagi}.

Similarly, \#68 of \cite[\S2.2 Table~1]{thamarai2} (= \#29696 in \cite{grdbweb})
\[
X_{\left(\begin{smallmatrix}3\\6\end{smallmatrix}\right)} = \TV_{\left(\begin{smallmatrix}2\\3\end{smallmatrix}\right)}\begin{pmatrix}1&1&1&1&1&0\\0&1&2&2&3&1\end{pmatrix}
\quad\textrm{with $g=6$ and $\cB=\{4\times\tfrac{1}2(1,1,1)\}$}
\]
has anticanonical embedding into $\PP(1^8,2^4)$ in codimension~$8$.
There is no prime Fano \tf\ that matches this Hilbert series.

\subsection{Higher Fano index}\label{sec!higherindex}
\begin{table}[h!]
\centering
\rotateplot{0.73\linewidth}{
\begin{tabular}{cc|CCCCC|CCCCC|CCCCC|CCCCC|CCCCC|CCCCC|CCCCC|CCCCC|}
\multicolumn{2}{c}{}&\multicolumn{40}{c}{Genus}\\
&&$-2$&$-1$&0&1&2&3&4&5&6&7&8&9&10&11&12&13&14&15&16&17&18&19&20&21&22&23&24&25&26&27&28&29&30&31&32&33&34&35&36&37\\
\cline{2-42}
\multirow{35}*{\rotatebox{-90}{Codimension}}
&1&&&&&&&&&&&&&&&&&&&&&&&&&&&&&&&&&&&&&&&&\\
&2&&1&3&&1&&&&&&&&&&&&&&&&&&&&&&&&&&&&&&&&&&&\\
&3&&&&&&&&1&&&&&&&&&&&&&&&&&&&&&&&&&&&&&&&&\\
&4&&&2&&1&&&&&&&&&&&&&&&&&&&&&&&&&&&&&&&&&&&\\
&5&&&5&2&&2&&&&&&&&&&&&&&&&&&&&&&&&&&&&&&&&&&\\
&6&&5&7&5&1&&&&1&&&&&&&&&&&&&&&&&&&&&&&&&&&&&&&\\
&7&&&4&&&2&&&&&&1&&&&&&&&&&&&&&&&&&&&&&&&&&&&\\
\cline{2-42}
&8&&6&12&8&2&1&3&&2&1&&&&&&&&&&&&&&&&&&&&&&&&&&&&&&\\
&9&&&7&&1&4&&&&2&&&&&&&&&&&&&&&&&&&&&&&&&&&&&&\\
&10&&10&15&13&11&&4&1&1&&&&1&&&&&&&&&&&&&&&&&&&&&&&&&&&\\
&11&&4&14&1&1&4&&5&&2&&&&&&1&&&&&&&&&&&&&&&&&&&&&&&&\\
&12&&6&18&20&12&1&8&&1&&5&&1&&&&&&&&&&&&&&&&&&&&&&&&&&&\\
&13&&7&14&3&2&15&&6&&3&&&&3&&&&&&&&&&&&&&&&&&&&&&&&&&\\
&14&&10&31&50&23&2&10&1&11&1&3&1&2&&&&2&&&&&&&&&&&&&&&&&&&&&&&\\
\cline{2-42}
&15&&11&28&6&3&25&4&12&&7&1&5&1&2&&&&&&1&&&&&&&&&&&&&&&&&&&&\\
&16&&12&24&64&37&1&25&2&12&&7&1&1&&3&&1&&&&&&&&&&&&&&&&&&&&&&&\\
&17&&18&11&3&4&44&3&16&1&17&2&8&1&4&&&&2&&&&&&&&&&&&&&&&&&&&&&\\
&18&&17&12&44&67&5&42&1&22&3&9&1&8&2&5&&2&1&1&&1&&&&&&&&&&&&&&&&&&&\\
&19&&13&5&&2&40&3&47&&21&&14&&3&&5&&3&&&&&&1&&&&&&&&&&&&&&&&\\
&20&&2&&&14&1&38&2&26&1&25&2&16&&7&1&2&&3&1&1&&&1&&&&&&&&&&&&&&&&\\
&21&&&&&&9&&16&1&24&3&19&2&19&&6&&4&&&&2&&&&&&&&&&&&&&&&&&\\
\cline{2-42}
&22&&&&&&&&&8&&14&3&13&1&12&&8&&4&1&2&&&&2&&&&&&&&&&&&&&&\\
&23&&&&&&&&&&1&&7&1&9&1&11&1&4&&8&&2&&1&&&&1&&&&&&&&&&&&\\
&24&&&&&&&&&&&&&1&&6&&8&1&7&1&4&&3&&3&&&&&&&&&&&&&&&\\
&25&&&&&&&&&&&&&&&&&&8&1&5&1&4&&&&2&&&&&&&&&&&&&&\\
&26&&&&&&&&&&&&&&&&&&&1&&4&&4&&1&&1&&1&&1&&&&&&&&&\\
&27&&&&&&&&&&&&&&&&&&&&&&1&1&3&2&3&&&&&&1&&&&&&&&\\
&28&&&&&&&&&&&&&&&&&&&&&&&&&1&&3&&2&&&&&&&&&&&\\
\cline{2-42}
&29&&&&&&&&&&&&&&&&&&&&&&&&&&&&&&2&&&&&&&&&&\\
&30&&&&&&&&&&&&&&&&&&&&&&&&&&&&&1&&&1&1&&&&&&&\\
&31&&&&&&&&&&&&&&&&&&&&&&&&&&&&&&&&&&&1&2&&&&\\
&32&&&&&&&&&&&&&&&&&&&&&&&&&&&&&&&&&&1&&&&&&\\
&33&&&&&&&&&&&&&&&&&&&&&&&&&&&&&&&&&&&&&&&&\\
&34&&&&&&&&&&&&&&&&&&&&&&&&&&&&&&&&&&&&&&&&\\
&35&&&&&&&&&&&&&&&&&&&&&&&&&&&&&&&&&&&&&&&&1\\
\cline{2-42}
\end{tabular}
}
\caption{Number of Hilbert series of semistable Fano~\tfs\ of index~$\iota\ge2$ listed by genus and estimated minimal codimension.}
\label{fig!index2}
\end{table}

Among all Fano~\tfs~$X$, there are some that have divisible canonical class,
and we indicate the Hilbert series of those in Table~\ref{fig!index2}.
The notion of divisibility we consider here is the following:
$-K_X = \iota A$ for a maximal integer~$\iota\ge 2$ and ample integral Weil divisor $A$.
(Note that $\iota$ is quite different from the Gorenstein index $i_X$ and the Fano index~$f_X$.)
The graded ring~$R(X,A) = \oplus_{m\ge0} H^0(X,mA)$
is Gorenstein as~$H^0(X,-K_X)\subset R(X,A)$~\cite[5.1.9]{GW1}.
Suzuki \cite{S04,BS2,BS} carries out the analysis to find a set of possible baskets~$\cB_\iota$
for each~$\iota$, with additional genus information when~$\iota\le 2$.
Again we impose the stability condition $-K_X^3\le-3K_Xc_2$.
The numbers of baskets (or basket--genus pairs for~$\iota\le2$) per index~$\iota\ge1$ is:
\vspace{0.5em}
\[
\small
\setlength{\arraycolsep}{2mm}
\begin{array}{r|cccccccccccccc}
\toprule
\iota&1&2&3&4&5&6&7&8&9&11&13&17&19\\
\textrm{$\#\cB_\iota$}&39550&1413&181&82&34&6&12&4&2&3&2&1&1\\
\bottomrule
\end{array}
\vspace{1em}
\]
There are no prime Fano~\tfs\ of indices~$\iota=12$, 14, 15, 16 or 18 by~\cite{S04},
and by applying more sophisticated techniques \cite{prokhorov10} also proves that there
is no prime Fano~\tf\ of index~$10$; cf.~also~\cite{PR}

The point here is that if~$-K_X=\iota A$, then the~$A$\nobreakdash-em\-bedded
model~$\Proj R(X,A)$ is usually much simpler
than the anticanonical ring.
The numbers
in low $A$-embedding codimension are:
\vspace{0.5em}
\begin{equation}\label{eq!f2}
\small{
\setlength{\arraycolsep}{2mm}
\begin{array}{r|cccccccccccc|c}
\toprule
\iota&2&3&4&5&6&7&8&9&11&13&17&19&\text{total}\\
\cmidrule(lr){1-1} \cmidrule(lr){2-13} \cmidrule(lr){14-14}
\text{codim 0}&0&0&1&1&0&1&0&0&1&1&1&1&7\\
\text{codim 1}&8&7&2&5&1&4&3&2&2&1&&&35\\
\text{codim 2}&26&6&7&1&&&&&&&&&40\\
\bottomrule
\end{array}}
\vspace{1em}
\end{equation}

\begin{example}\label{eg!toricwps}
The codimension~$0$ row of~\eqref{eq!f2} contains
the~$7$ weighted projective spaces
\[
\small
\vspace{0.5em}
\setlength{\arraycolsep}{1.5mm}
\begin{array}{r|ccccccc}
\toprule
\multicolumn{1}{r}{}&\PP^3&\PP(1^3,2)&\PP(1^2,2,3)&\PP(1,2,3,5)&\PP(1,3,4,5)&\PP(2,3,5,7)&\PP(3,4,5,7)\\
\cmidrule(lr){2-8}
{-}K^3&64&125/2&343/6&1331/30&2197/60&4913/210&6859/420\\
g&33&32&29&22&18&11&7\\
\text{cod}&31&31&30&27&25&22&20\\
W&1^{35}&1^{34},2&1^{31},2^2,3&
1^{24},2^4,3^2,5&
1^{20},2^4,3^3,4,5&
1^{13},2^5,3^4,4^2,5,7&
1^{9},2^7,3^4,4,5^2,7\\
\bottomrule
\end{array}
\]
and these embed anticanonically in~$\PP(W)$ as the model given in \grdb.
\end{example}

\begin{example}\label{eg!veronese}
Many of the~$A$-hypersurfaces embed in a bigger space under their anticanonical~$\iota$\nobreakdash-Vero\-nese
embedding than \grdb\ suggests.
For example,~$X_{10}\subset\PP(1^2,2,3,5)$ has~$\iota=2$,
so~$-K_X=2A$ and the anticanonical embedding is
\[
X\buildrel{\cong}\over{\longrightarrow}\Phi_{2A}(X)\subset\PP(1^4,2^2,3^3,4)
\]
in codimension~$6$. Its ($-K_X$-polarised) Hilbert series, however, admits a simpler model~$Y_{4,4}\subset\PP(1^4,2,3)$,
and this is the one given in \grdb.
\end{example}

There is a small twist in the \grdb\ numerics in this case.
For varieties with $\iota\ge2$, Gorenstein projection does not work in the same way as the~$\iota=1$ case
(the image of projection has non-isolated singularities).
Therefore rather than using projection to identify weights, the \grdb\ compares with the K3 database~\cite{k3db}: a degree~$\iota$ slice
$S\in|\iota A|$ has the numerical properties of a K3 surface polarised
by~$A_{|S}$, so a choice of weights is determined by~\cite{k3db};
then including a variable of degree~$\iota$ (noting that it may
then be eliminated by an equation of degree~$\iota$) gives a choice of weights for~$X$.

\subsection{Toric Fano 3-folds}\label{sec!toric}
Kasprzyk~\cite{kasprzyk10} classifies toric Fano~\tfs\ with canonical singularities
as a list~$\Tcan$ of~$674,\!688$ lattice polytopes; this classification is available from~\cite{canonicaldata}.
These varieties are all~$\Q$\nobreakdash-Goren\-stein, and $12,\!190$ of them are~$\Q$\nobreakdash-factorial.

The \grdb\ contains~$\Tcan$, and it can be analysed online~\cite{grdbweb}.
More importantly, the \grdb\ connects~$\Tcan$ with the Fano~\tf\ database:
each polytope~$\De\in\Tcan$ has a Hilbert series~$P_\De\in\FMF$,
all but~$12$ of which lie in~$\Fss$.
These absent~$12$ all have isolated,~$\Q$\nobreakdash-factorial, strictly canonical singularities,
and are either one of the six weighted projective spaces
\[
\PP(1, 1, 3, 5),\
\PP(1, 2, 5, 7),\
\PP(1, 3, 7, 10),\
\PP(1, 3, 7, 11),\
\PP(1, 5, 7, 13),\
\PP(3, 5, 11, 19)
\]
or a rank~$2$ blowup of one of these.
The \grdb\ links~$\De$ to~$P_\De$,
and conversely for any~$P\in\Fss$ reports all those~$\De$ with~$P_\De=P$.
This matching is significant in different ways, as we discuss next.

\subsubsection{Location in geography}
The toric Fano~\tfs\ in~$\Tcan$ populate large areas of Table~\ref{fig!geography}.
Table~\ref{fig!cantoricnumbers} is the submap of Hilbert series that are realised by at least one toric Fano~\tf.

Only eight of the toric cases are prime Fano~\tfs: the seven weighted
projective spaces of Example~\ref{eg!toricwps} together with the magical fake
weighted projective space of degree~$64/5$ and genus~$5$
\[
\PP^3/\bmu_5(1,2,3,4) \hookrightarrow \PP(1^7, 2^8, 3^4, 5^4)
\]
in codimension~$19$. Including these, there are~$634$ terminal cases, of
which~$233$ are terminal~$\Q$\nobreakdash-factorial.

For some Hilbert series there are many matching toric \tfs, and this is
recorded in the \grdb, with an idea of the multiplicities in
Table~\ref{fig!cantoric}. It seems amazing to us that two different
polytopes can contain the same number of lattice points at all dilations
-- but of course this is exactly what happens whenever two polytopes are
mutation equivalent~\cite{ACGK}.

There are~$4319$ varieties~$X\in\Tcan$ that have~${-}K_X$ Cartier (cf.~\cite{KS98});
of these,~$194$ are~$\Q$\nobreakdash-factorial.
The number of these is listed by genus in the column~$\#\cT$ of Table~\ref{fig!MM}.
We have not checked whether any satisfy Petracci's non-smoothability condition~\cite[1.1]{petracci20},
nor whether any lie at the intersection of multiple
smooth families (see~\S\ref{sec!toricdef}).

\subsubsection{Toric degenerations}
Some approaches to or applications of Mirror Symmetry require toric degenerations of Fano \tfs~\cite{CI,CCGK}.
The link between lists $\Fss$ and $\Tcan$ summarised in Table~\ref{fig!cantoricnumbers}
is a necessary condition for a Fano~\tf~$X$ with given Hilbert series to
have a toric Fano~\tf\ degeneration~$X_0$. This numerical condition is not sufficient,
as it does not determine whether~$X$ and~$X_0$ lie in the same deformation
family, for example.

It would be natural to extend this correspondence either to include more general
toric varieties, or reducible varieties composed of toric varieties
glued along toric strata.

\subsubsection{Deformation of toric varieties and intersections of families}\label{sec!toricdef}
Following Altmann's local analysis~\cite{altmann1,altmann2}, a lot is known about how
toric varieties deform.
For toric Fano~\tfs\ with isolated singularities,
global deformations surject onto local deformations~\cite[2.3]{petracci19},
so understanding the deformation theory of singularities on toric Fano~\tfs\ is a
powerful tool.

\begin{example}\label{eg!def1}
The first element~$X_1\in\Tcan$ has the Hilbert series of some~$X\subset\PP(1^{7},2^4)$
with a basket~$4\times\tfrac{1}2(1,1,1)$.
However,~$X_1$ is not quasismooth: it has six Gorenstein facets that are the cone on the del~Pezzo surface of degree~$6$ and four cones of type~$\tfrac{1}2(1,1,1)$.
Each of these del~Pezzo cone singularities has two smoothing components locally,
so since deformations of toric Fano~\tfs\ surject onto local deformations,
there are at least seven distinct quasi-smoothing components that contain
different small deformations of~$X_1$ (cf.\ \ref{eg!fanosearch} below which suggests there are at least 24 families).

Extending this to Gorenstein-index-$2$ toric Fano~\tfs\ that also have isolated
cone over del~Pezzo degree~$6$ singularities gives three more examples with two distinct
quasi-smoothing families:
\[
\small
\setlength{\arraycolsep}{1mm}
\begin{array}{ccccccc}
\toprule
\Tcan(\textrm{id})&\#\tfrac{1}2&\#\text{dP}_6&\Fss(\textrm{id})&X\subset w\PP&g&\text{codim}\\
\cmidrule(lr){1-1}\cmidrule(lr){2-3}\cmidrule(lr){4-4}\cmidrule(lr){5-7}
1&4&6&27334&\PP(1^7,2^4)&5&7\\
254482&3&1&38250&\PP(1^{17},2^3)&15&16\\
254485&6&1&36639&\PP(1^{13},2^6)&11&15\\
254810&3&1&38935&\PP(1^{20},2^3)&18&19\\
\bottomrule
\end{array}
\]
\end{example}

\subsubsection{High codimension representatives and cascades}\label{sec!cascades}
Many of the weights in \grdb\ are constructed inductively by considering a single projection,
but varieties frequently arise in sequences, or~\emph{cascades}, of projections:
famously~$X = \Phi_{-K}(\PP^2)\subset\PP^9$ has sequences of
projections from points (which in this case are blowups) that recover
elements of most families of del~Pezzo surfaces (see~\cite{RS} for extensions).
It seems typical that the ends of such cascades are simpler to describe
than the middles: things like toric varieties live at the top, while hypersurfaces live
at the bottom.
Thus any~$X\in\Tcan$ of high codimension for its genus may be a good candidate for
the head of a cascade.

For example, in~$g=8$, the highest codimension~$\Q$\nobreakdash-factorial terminal
Fano~\tf\ is
\[
\Tcan(544385)\colon
\PP^1\times\PP^2 / {\tfrac{1}3(0,1,0,1,2)} \subset\PP(1^{10},2^6,3^6)
\]
of Picard rank~$2$ in codimension~$18$, with~$6\times\tfrac{1}3(1,1,2)$ singularities
at the 6 index~3 toric 0\nobreakdash-strata. (This has id \#33967 in $\Fss$.)
The Fano polytope is the simplicial decomposition on vertices
\[
(1,0,0), (0,1,0), (-1,-1,0), (1,2,3), (-1,-2,-3)
\]
with six index~$3$ cones meeting at a central `$\PP^2$' triangular equator with
a cycle of three northern cones and three southern cones with~$\Sym_3$ symmetry.
(This is a lattice triangular bipyramid and is fun to draw on a blackboard.)
The Kawamata blowup of any one of the
index~$3$ points is equivalent to any other, and gives the first projection.
There are four ways to project from a pair of index~$3$ points, depending on
adjacency; in the case of the blowup of a northern cone and the adjacent
southern cone, the equator becomes a flopping curve that is contracted to an ordinary node.
Continuing with different configurations of projections, we end up at the projection from all index~$3$ points, which is a variety $Y\subset\PP(1^{10},2^6)$ in codimension~$12$ with~$6\times\tfrac{1}2(1,1,1)$ singularities and 3 nodes.
This variety admits a quasi-smoothing, so is a Gorenstein-index-$2$ Fano~\tf\
that does not appear in Takagi's classification Table~\ref{fig!takagi}, as it has~$\rho_Y>1$.
Further projections from index~$2$ points give more Gorenstein-index-$2$ varieties
that extend Table~\ref{fig!takagi} in genus~$8$ to Fano~\tfs\ that are not prime Fano.

The highest-codimension toric Fano~\tfs\ with terminal singularities by genus -- that is,
the toric candidates for the top of terminal cascades -- are listed in Table~\ref{fig!projtops}.

Beyond toric,~\cite{hausen1,hausen2} initiates the analysis of low complexity Fano varieties with the classification of prime Fano~\tfs\ of complexity~$1$ -- that is, Fano \tfs\ whose automorphism group has maximal torus of rank~2. 
The result is a list of 46 varieties and one 1\nobreakdash-parameter family. The complexity
condition enforces very particular trinomial equations.
As in Example~\ref{eg!veronese}, these are almost all toric
hypersurfaces with high-codimension anticanonical embedding.
For example, the family of hypersurfaces $X_6\subset\PP(1,2,3,4,5)$ have $\cO_X(-K_X) = \cO_X(9)$,
genus $g=18$ and anticanonical embedding in codimension~24.
It contains two particular members, $(x_0x_4 + x_1x_3 + x_2^2=0)$ and $(x_0x_4 + x_1^3 + x_2^2=0)$, which have complexity~1
(the latter is not quasismooth, but has an exceptional terminal singularity of type $\tfrac{1}4(1,1,2,3;2)$).

\begin{table}[t]
\[
\small
\setlength{\arraycolsep}{2mm}
\begin{array}{ccccccccl}
\toprule
\Tcan(\textrm{id})&g&\rho_X&\Fss(\textrm{id})&X\subset w\PP&\cB&\text{codim}&\\
\cmidrule(lr){1-1}\cmidrule(lr){2-3}\cmidrule(lr){4-4}\cmidrule(lr){5-7}
547383&5&1&29211&\PP(1^7,2^8,3^4,5^4)&4\times\tfrac{2}5&19&\PP^3/\tfrac{1}5(1,2,3,4)\\
547379&7&1&32734&\PP(1^9,2^7,3^4,4,5^2,7)&\tfrac{1}3,\tfrac{1}4,\tfrac{2}5,\tfrac{3}7&20&\PP(3,4,5,7)\\
544385&8&2&33967&\PP(1^{10},2^6,3^6)&6\times\tfrac{1}3&18&\\
547380&11&1&36623&\PP(1^{13},2^5,3^4,4^2,5,7)&\tfrac{1}2,\tfrac{1}3,\tfrac{1}5,\tfrac{3}7&22&\PP(2,3,5,7)\\
430483&12&4&36948&\PP(1^{14},2^3,3^3)&3\times\tfrac{1}3&16&\\
520102&13&3&37585&\PP(1^{15},2^4,3^3,4)&\tfrac{1}2,2\times\tfrac{1}3,\tfrac{1}4&19&\\
544370&14&2&38020&\PP(1^{16},2^4,3^3,4,5)&\tfrac{1}3,\tfrac{1}4,\tfrac{2}5&21&\\
544376&15&2&38404&\PP(1^{17},2^5,3^2,5)&2\times\tfrac{1}2,\tfrac{1}3,\tfrac{2}5&21&\\
430473&16&4&38533&\PP(1^{18},2^3,3)&2\times\tfrac{1}2,\tfrac{1}3&18&\\
520107&17&3&38760&\PP(1^{19},2^3,3)&2\times\tfrac{1}2,\tfrac{1}3&19&\\
520148&17&2&38760&\text{same as 520107}\\
547382&18&1&39006&\PP(1^{20},2^4,3^3,4,5)&\tfrac{1}3,\tfrac{1}4,\tfrac{2}5&25&\PP(1,3,4,5)\\
520124&19&3&39052&\PP(1^{21},2^2,3)&\tfrac{1}2,\tfrac{1}3&20&\\
520103&20&3&39192&\PP(1^{22},2^3,3)&2\times\tfrac{1}2,\tfrac{1}3&22&\\
544394&21&1&39278&\PP(1^{23},2^3,3)&2\times\tfrac{1}2,\tfrac{1}3&23&X_4\subset\PP(1^2,2^2,3)\\
547381&22&1&39368&\PP(1^{24},2^4,3^2,5)&\tfrac{1}2,\tfrac{1}3,\tfrac{2}5&27&\PP(1,2,3,5)\\
520128&24&3&39416&\PP(1^{26},2)&\tfrac{1}2&23\\
520131&24&3&39416&\text{same as 520128}\\
544383&25&2&39457&\PP(1^{27},2^2,3)&\tfrac{1}2,\tfrac{1}3&26&\\
544389&26&2&39476&\PP(1^{28},2)&\tfrac{1}2&25&\\
544388&28&2&39510&\PP(1^{30},2)&\tfrac{1}2&27&\\
547384&29&1&39526&\PP(1^{31},2^2,3)&\tfrac{1}2,\tfrac{1}3&30&\PP(1^2,2,3)\\
547385&32&1&39541&\PP(1^{34},2)&\tfrac{1}2&31&\PP(1^3,2)\\
\bottomrule
\end{array}
\]
\caption{High-codimension non-Gorenstein~$\Q$-factorial terminal toric Fano~\tfs, sorted by increasing genus $g$.}
\label{fig!projtops}
\end{table}

\begin{table}[h!]
\centering
\rotateplot{0.79\linewidth}{
\begin{tabular}{cc|CC|CCCCC|CCCCC|CCCCC|CCCCC|CCCCC|CCCCC|CCCCC|}
\multicolumn{2}{c}{}&\multicolumn{37}{c}{Genus}\\
&&$-2\,\cdots$&2&3&4&5&6&7&8&9&10&11&12&13&14&15&16&17&18&19&20&21&22&23&24&25&26&27&28&29&30&31&32&33&34&35&36&37\\
\cline{2-39}
\multirow{35}*{\rotatebox{-90}{Codimension}}
&1&&&1&&&&&&&&&&&&&&&&&&&&&&&&&&&&&&&&&&\\
&2&&&1&1&&&&&&&&&&&&&&&&&&&&&&&&&&&&&&&&&\\
&3&&&2&1&1&&&&&&&&&&&&&&&&&&&&&&&&&&&&&&&&\\
&4&&&3&2&1&1&&&&&&&&&&&&&&&&&&&&&&&&&&&&&&&\\
&5&&&6&3&2&1&1&&&&&&&&&&&&&&&&&&&&&&&&&&&&&&\\
&6&&&7&6&3&2&1&1&&&&&&&&&&&&&&&&&&&&&&&&&&&&&\\
&7&&&12&8&6&3&2&1&1&&&&&&&&&&&&&&&&&&&&&&&&&&&&\\
\cline{2-39}
&8&&&15&15&8&6&3&2&1&1&&&&&&&&&&&&&&&&&&&&&&&&&&&\\
&9&&&23&20&15&8&6&3&2&1&1&&&&&&&&&&&&&&&&&&&&&&&&&&\\
&10&&&28&32&21&15&8&6&3&2&1&1&&&&&&&&&&&&&&&&&&&&&&&&&\\
&11&&&40&43&35&21&15&8&6&3&2&1&1&&&&&&&&&&&&&&&&&&&&&&&&\\
&12&&1&45&65&47&35&21&15&8&6&3&2&1&1&&&&&&&&&&&&&&&&&&&&&&&\\
&13&&&59&86&71&48&35&21&15&8&6&3&2&1&1&&&&&&&&&&&&&&&&&&&&&&\\
&14&&&54&126&92&72&48&35&21&15&8&6&3&2&1&1&&&&&&&&&&&&&&&&&&&&&\\
\cline{2-39}
&15&&&51&148&120&89&74&48&35&21&15&8&6&3&2&1&1&&&&&&&&&&&&&&&&&&&&\\
&16&&&33&161&127&107&93&72&46&35&21&15&8&6&3&2&1&1&&&&&&&&&&&&&&&&&&&\\
&17&&&15&100&102&84&128&87&69&46&35&21&15&8&6&3&2&1&1&&&&&&&&&&&&&&&&&&\\
&18&&&3&33&37&32&131&112&84&66&42&33&21&15&8&6&3&2&1&1&&&&&&&&&&&&&&&&&\\
&19&&&1&4&10&2&122&88&95&61&54&33&33&20&15&8&6&3&2&1&1&&&&&&&&&&&&&&&&\\
&20&&&&1&1&&49&52&58&29&32&21&28&25&21&15&8&6&3&2&1&1&&&&&&&&&&&&&&&\\
&21&&&&&&&11&10&17&2&11&3&9&11&28&18&14&7&6&3&2&1&1&&&&&&&&&&&&&&\\
\cline{2-39}
&22&&&&&&&1&&2&&1&&&1&13&25&14&13&8&6&3&2&1&1&&&&&&&&&&&&&\\
&23&&&&&&&&&&&&&&&3&14&15&14&15&8&6&3&2&1&1&&&&&&&&&&&&\\
&24&&&&&&&&&&&&&&&&6&1&9&17&13&6&6&3&2&1&1&&&&&&&&&&&\\
&25&&&&&&&&&&&&&&&&1&&3&17&10&6&7&5&3&2&1&1&&&&&&&&&&\\
&26&&&&&&&&&&&&&&&&&&&5&5&1&12&6&6&2&2&1&1&&&&&&&&&\\
&27&&&&&&&&&&&&&&&&&&&&1&&12&7&6&1&2&2&1&1&&&&&&&&\\
&28&&&&&&&&&&&&&&&&&&&&&&1&2&3&&1&2&2&1&1&&&&&&&\\
\cline{2-39}
&29&&&&&&&&&&&&&&&&&&&&&&&&&&&&&2&1&1&&&&&&\\
&30&&&&&&&&&&&&&&&&&&&&&&&&&&&&&1&1&1&1&&&&&\\
&31&&&&&&&&&&&&&&&&&&&&&&&&&&&&&&&1&1&1&&&&\\
&32&&&&&&&&&&&&&&&&&&&&&&&&&&&&&&&1&&&1&&&\\
&33&&&&&&&&&&&&&&&&&&&&&&&&&&&&&&&&&&1&&&\\
&34&&&&&&&&&&&&&&&&&&&&&&&&&&&&&&&&&&&&&\\
&35&&&&&&&&&&&&&&&&&&&&&&&&&&&&&&&&&&&&&1\\
\cline{2-39}
\end{tabular}
}
\caption{The 5760 Hilbert series of (canonical) toric Fano \tfs\  that lie in $\Fss$,
listed by genus and codimension. Note that 12 other toric Fano Hilbert series lie in $\FMF$ outside $\Fss$.}
\label{fig!cantoricnumbers}
\end{table}

\subsection{Formats and low codimension}\label{sec!formats}
All~$95+85$ Hilbert series in~$\FMF$ whose \grdb\ model is in codimension~$1$ or~$2$
actually lie in~$\Fss$ and may be constructed by hand as complete intersections as proposed;
these are the first two rows of Table~\ref{fig!geography}.
These varieties all have Picard rank~$1$ by the Lefschetz hyperplane theorem (cf.\ \cite[3.5]{cpr}).

The same is true of all~$70$ Hilbert series with codimension~$3$ models, which occupy the
third row of Table~\ref{fig!geography}.
In that case, only~$X_{2,2,2}\subset\PP^6$ is a complete intersection.
The remaining~$69$ cases are cut out by the five maximal Pfaffians of a skew~$5\times5$
matrix. Corti and Reid~\cite{wgrass}, following Grojnowski, explain this as a pullback from
a weighted Grassmannian~$\wGrass(2,5)$ in a precise sense, which informally
we may treat as saying that the Pl\"ucker embedding~$\Grass(2,5)\subset\PP^9$
is described by the Pfaffians of a generic skew~$5\times5$ matrix of linear forms,
and we may specialise these forms as we please, taking care with homogeneity.
Again, these varieties have Picard rank~$1$ by~\cite[3]{BF}.

This idea leads to the general idea of `format'~\cite{formats}, where the equations (and syzygies, and indeed the whole minimal free resolution) of a `key variety' (that is, of any variety you like) are used as a model for the equations of other varieties by graded pullback.
This idea is implemented in several independent programmes: for example, early cases include \cite{wgrass,qs}, Coughlan--Ducat apply cluster varieties \cite{CD}, and Takagi gives a sequence of key varieties in higher codimension~\cite{takagi21a,takagi21b,takagi23b,takagi24a,takagi24b,takagi24c}. 
These case construct prime Fano \tfs, but other format constructions such as \cite{BKQ} may lead to higher Picard rank.

\begin{table}[h!]
\centering
\rotateplot{0.59\linewidth}{
\begin{tabular}{cc|CC|WWWWW|WWWWW|WWWWW|WWWWW|WWWWW|WWWWW|WWWWW|}
\multicolumn{2}{c}{}&\multicolumn{37}{c}{Genus}\\
&&$-2\,\cdots$&2&3&4&5&6&7&8&9&10&11&12&13&14&15&16&17&18&19&20&21&22&23&24&25&26&27&28&29&30&31&32&33&34&35&36&37\\
\cline{2-39}
\multirow{35}*{\rotatebox{-90}{Codimension}}&1&&&1&&&&&&&&&&&&&&&&&&&&&&&&&&&&&&&&&&\\
&2&&&1&7&&&&&&&&&&&&&&&&&&&&&&&&&&&&&&&&&\\
&3&&&3&10&23&&&&&&&&&&&&&&&&&&&&&&&&&&&&&&&&\\
&4&&&8&55&50&54&&&&&&&&&&&&&&&&&&&&&&&&&&&&&&&\\
&5&&&32&162&307&182&135&&&&&&&&&&&&&&&&&&&&&&&&&&&&&&\\
&6&&&76&584&921&927&422&207&&&&&&&&&&&&&&&&&&&&&&&&&&&&&\\
&7&&&231&1468&3083&2705&1980&745&314&&&&&&&&&&&&&&&&&&&&&&&&&&&&\\
\cline{2-39}
&8&&&490&3598&6537&7282&4971&2954&1009&373&&&&&&&&&&&&&&&&&&&&&&&&&&&\\
&9&&&1025&6084&12835&13094&10904&6230&3372&1121&416&&&&&&&&&&&&&&&&&&&&&&&&&&\\
&10&&&1493&9308&17594&20444&15942&11726&6365&3366&1142&413&&&&&&&&&&&&&&&&&&&&&&&&&\\
&11&&&1897&10196&21656&22519&19992&14512&10109&5439&3068&1046&413&&&&&&&&&&&&&&&&&&&&&&&&\\
&12&&7&1739&9926&17859&21825&18423&15498&10568&7540&4321&2498&904&348&&&&&&&&&&&&&&&&&&&&&&&\\
&13&&&1302&7133&13379&15659&14817&12117&10204&7041&5492&3292&2017&744&334&&&&&&&&&&&&&&&&&&&&&&\\
&14&&&637&4727&7202&9911&9294&8810&7628&5982&4467&3816&2417&1529&622&274&&&&&&&&&&&&&&&&&&&&&\\
\cline{2-39}
&15&&&262&2317&3599&4356&5580&5063&5694&3699&3602&2912&2535&1652&1272&486&234&&&&&&&&&&&&&&&&&&&&\\
&16&&&72&989&1210&1600&2818&3139&3581&2291&2133&2188&1885&1672&1398&950&386&179&&&&&&&&&&&&&&&&&&&\\
&17&&&21&253&389&369&1763&1656&2235&1155&1335&1131&1274&1080&1491&953&657&274&151&&&&&&&&&&&&&&&&&&\\
&18&&&3&51&80&64&751&895&1022&518&532&517&644&601&938&986&668&455&186&117&&&&&&&&&&&&&&&&&\\
&19&&&1&4&16&2&276&302&390&164&194&135&254&241&621&690&645&412&360&147&87&&&&&&&&&&&&&&&&\\
&20&&&&1&1&&68&81&98&36&47&33&53&78&285&557&441&364&300&225&91&66&&&&&&&&&&&&&&&\\
&21&&&&&&&13&10&21&2&11&3&10&12&106&264&252&187&323&168&120&72&40&&&&&&&&&&&&&&\\
\cline{2-39}
&22&&&&&&&1&&2&&1&&&1&20&119&85&96&199&185&103&110&32&42&&&&&&&&&&&&&\\
&23&&&&&&&&&&&&&&&3&32&28&42&182&103&72&80&61&27&27&&&&&&&&&&&&\\
&24&&&&&&&&&&&&&&&&9&1&13&81&68&29&68&29&27&11&18&&&&&&&&&&&\\
&25&&&&&&&&&&&&&&&&1&&3&34&17&11&30&26&19&13&16&8&&&&&&&&&&\\
&26&&&&&&&&&&&&&&&&&&&5&6&1&25&15&18&3&15&5&13&&&&&&&&&\\
&27&&&&&&&&&&&&&&&&&&&&1&&17&10&10&1&5&4&3&9&&&&&&&&\\
&28&&&&&&&&&&&&&&&&&&&&&&2&2&3&&1&2&3&2&4&&&&&&&\\
\cline{2-39}
&29&&&&&&&&&&&&&&&&&&&&&&&&&&&&&2&4&2&&&&&&\\
&30&&&&&&&&&&&&&&&&&&&&&&&&&&&&&1&2&1&2&&&&&\\
&31&&&&&&&&&&&&&&&&&&&&&&&&&&&&&&&1&1&5&&&&\\
&32&&&&&&&&&&&&&&&&&&&&&&&&&&&&&&&1&&&1&&&\\
&33&&&&&&&&&&&&&&&&&&&&&&&&&&&&&&&&&&1&&&\\
&34&&&&&&&&&&&&&&&&&&&&&&&&&&&&&&&&&&&&&\\
&35&&&&&&&&&&&&&&&&&&&&&&&&&&&&&&&&&&&&&2\\
\cline{2-39}
\end{tabular}
}
\caption{Number of (canonical) toric Fano \tfs\ that satisfy the semistable numerical condition,
listed by genus and codimension.}
\label{fig!cantoric}
\end{table}

\begin{example} \label{eg!differentcodim}
In~\cite[1.2]{BKQ}, the variety~$\PP^2\times\PP^2\subset\PP^8$ in its Segre embedding
is used as a key variety to model some Fano~\tfs\ in codimension~$4$ that have Picard rank~$2$.
That analysis constructs examples of deformation families in different codimension
for the same Hilbert series.
For example,~$\Fss(548)$ is presented in \grdb\ as
\[
X\subset\PP(1,3,4,5,6,7,10)
\]
a codimension~$3$ Pfaffian that is easy to construct, but there is another family
\[
X'\subset\PP(1,3,4,5,6,7,9,10)
\]
which arises from unprojection of a degeneration.
One may suspect that such Fano~\tfs\ are degenerations of the Pfaffian model,
but this is not the case: quasismooth members of the two families have different
Picard rank~$h^{1,1}(X)$ and
Euler characteristic~$e_X$, so do not lie in a common flat family.
\end{example}

Different formats cover everything in codimension~$\le3$,
much in codimension~$4$ (see~\cite[5.3]{CD}, where a cluster
variety format describes certain subfamilies of deformations),
and some in codimension~$5$.
Although there are examples in higher codimension, they seem to
realise only a small part of the classification there -- see~\cite[4.4, 5.2]{qs}, where
flag variety formats recover classical nonsingular Fano~\tfs\ but no other element
of~$\Fss$ has suitable Hilbert series, or~\cite[5.8]{CD}, where a codimension~$6$ format
realises no Fano~\tfs. (Of course, such failures could be because
there are few Fano~\tfs\ in high codimension -- it seems unlikely, but we simply do not know.)

\subsection{Codimension 4: unprojection}\label{sec!unprojection}
The unprojection Ansatz in~\S\ref{sec!ansatz} can sometimes be realised. For example,
given a coordinate plane $D=\PP(a_i,a_j,a_k)\subset\PP(a_0,\dots,a_n)$,
one may be able to construct a Fano \tf~$X$ that contains~$D$ with~$X$ quasismooth
away from finitely many nodes on~$D$.
In this model case, the Type~I unprojection constructs
a quasismooth Fano~\tf~$Y\subset\PP(a_0,\dots,a_n,r)$ (the mild
numerical conditions of~\S\ref{sec!unproj} are only to exclude quasi-linear
equations in the ideal of~$Y$).

This is carried out in~\cite{tj} in the case~$X\subset \wps^6$
lies in codimension~$3$, so that~$Y\subset\wps^7$ lies in codimension~$4$,
with further cases in~\cite{BCD2}.
Type~II unprojections are considered in~\cite{kinosaki,papadakis08,Taylor}, which construct other
cases in codimension~$4$.
Using this and~\S\ref{sec!formats}, for every Hilbert series of~$\Fss$ whose
\grdb\ model is in codimension~$4$, there is a construction of a variety that matches that
model, and in the majority of cases there are 2 or more distinct deformation families.

\subsection{The Fanosearch programme}\label{sec!fanosearch}
Coates, Corti, Galkin, Golyshev, Kasprzyk~\cite{CCGGK13} and others, following ideas of Golyshev~\cite{golyshev},
provide an alternative approach to the Fano classification problem.
The idea is that, via Mirror Symmetry, Fano classification can be rephrased as a fundamentally combinatorial problem of identifying suitable Laurent polynomials whose periods
generate solutions to Picard--Fuchs equations on the other side of the mirror.
Although the combinatorial problems seem hard, and the required mirror theorems
are not wholly in place,~\cite{CCGK} confirms that the two sides of the mirror agree
in the smooth case,
and in doing so provides a host of tools for constructing Fano varieties
and passing through the mirror. The \grdb\ is a key tool (see e.g.~\cite{CHKP,Heu,dialogues}).

\begin{example}\label{eg!fanosearch}
Let~$P$ be the Fano polytope whose spanning fan gives rise to~$X_1\in\Tcan$. Using the terminology of~\cite{CKPT21},~$P$ supports eleven rigid maximally mutable Laurent polynomials (rigid MMLPs), up to automorphisms of~$P$. The periods of these rigid MMLPs give solutions to eleven distinct Picard--Fuchs equations; compare this with Example~\ref{eg!def1}, where prima facie we see seven deformation families. By~\cite[Conjecture~5.1]{CKPT21} we expect each of these eleven rigid MMLPs to correspond to a deformation of~$X_1$ to a terminal locally toric Fano with~$4\times\tfrac{1}2(1,1,1)$ singularities. This expectation agrees with the output of Ilten's Macaulay2 package~\cite{Ilten}. More generally, considering those~$X\in\Tcan$ with Hilbert series equal to that of~$X_1$ gives a total of~$24$ rigid MMLPs, up to mutation, corresponding to~$24$ distinct Picard--Fuchs equations and hence, conjecturally, at least~$24$ deformation families of terminal Fano\ \tfs\ with basket~$4\times\tfrac{1}2(1,1,1)$.
\end{example}

\section{Synopsis}\label{sec!summary}
\subsection{Review of guiding examples}\label{sec!examples}
The Fano~\tf\ database consists of two lists of Hilbert series~$\Fss\subset\FMF$ (defined in Definition~\ref{def!Fss})
together with the estimated weights~$X\subset\PP(a_0,\dots,a_n)$ that the \grdb\ assigns to each one (\S\S\ref{sec!lowcodim} and~\ref{sec!unproj}).
It is intended as a first coarse approximation to the classification of Fano~\tfs.
However, it is certainly nowhere near to a final classification.
The following remarks and examples are intended as quick reminders
to help avoid misunderstandings.

\begin{enumerate}
\item
Overview of the Fano~\tf\ database:
\vspace{0.5em}
\begin{enumerate}
\item
We distinguish between Fano \tfs\ (with canonical singularities; Definition~\ref{def!fano})
and the much smaller class of prime Fano~\tfs\ (outcomes of the Minimal Model Program;
Definition~\ref{def!prime}).
\item\label{5b}
The set~$\Fss$ contains all rational functions that satisfy necessary conditions
to be the Hilbert series of a Fano~\tf\ $X$ with $-K_X^3\le-3K_Xc_2$.
The geography of Table~\ref{fig!geography} is of~$\Fss$ only.
\item
By \cite{kawamata}, the inequality in~\eqref{5b} is satisfied by semistable prime Fano~\tfs\ (\S\ref{sec!building})
and \cite{LL25} proves it holds for all prime Fano \tfs\ (\ref{cor!ss}).
\item
A series~$P\in\Fss$ may be realised by many deformation families
(Tables~\ref{fig!MM} and~\ref{fig!takagi}) or by none~(\ref{thm!karzh66}\eqref{X66ii}): $\Fss$ does not count the number of deformation families.
This is a basic part of the classification problem, and it is fully understood
only in the case of nonsingular Fano~\tfs\ and some specific cases with only~$\tfrac{1}2(1,1,1)$ singularities.
\end{enumerate}
\item[]
\item
Existence and non-existence:
\vspace{0.5em}
\begin{enumerate}
\item
Proven cases of prime Fano~\tfs\ are sparse: the nonsingular Fano~\tfs\ of Picard rank~$1$
(\S\ref{sec!smooth}),
Takagi's Gorenstein-index-$2$ classification (\S\ref{sec!index2}),
and a range of cases in low anticanonical
codimension (\S\ref{sec!formats}) or of high Fano index~(\S\ref{sec!higherindex}).
\item
There is no reason why the Hilbert series of more general Fano~\tfs\
should appear in~$\Fss$ or~$\FMF$, yet every Fano~\tf\ we know does in fact appear on these lists.
\item
Many locations in Table~\ref{fig!geography} are realised
by a Fano~\tf\ but not by a prime Fano~\tf\ (Table~\ref{fig!MM} and \S\S\ref{sec!gorenstein}, \ref{sec:nonprime2}, and~\ref{sec!cascades}).
\item
We expect that many of the high codimension, lower genus Hilbert series are
not realised by any Fano~\tf.
However, we only know one place in Table~\ref{fig!geography}
where this is proven: Karzhemanov's nonexistence result for genus~$35$~(\S\ref{sec!gorenstein}).
\end{enumerate}
\item[]
\item
The estimated anticanonical embedding~$X\subset\PP(a_0,\dots,a_n)$:
\vspace{0.5em}
\begin{enumerate}
\item
The Hilbert series~$P_X$ of a Fano~\tf~$X$ does not determine the weights
$a_0,\dots,a_n$ of its anticanonical embedding (Example~\ref{eg!differentcodim}).
\item
The embedding weights given to each~$P\in\FMF$ in the \grdb\ are only a suggestion.
They are derived from known examples in low
codimension~(\S\ref{sec!famous95}), an analysis of conjectured Gorenstein
projections~(\S\ref{sec!unproj}) and, in harder cases, an
analysis of singularities or the linear systems on possible K3 sections.
Although the weights are often right, there is no reason why your~$X$ should be embedded
in this way.
\item
Even if the general member of a deformation family is~$X\subset\PP(a_0,\dots,a_n)$,
there are likely to be degenerations in higher codimension~(\S\ref{sec!famous95}).
\item
It can happen that~$P\in\Fss$ is realised by two or more distinct deformation families whose
general members embed in different codimensions~(Example~\ref{eg!differentcodim}).
\item
Some higher index Fano~\tfs\ lie in higher codimension than \grdb\ suggests~(Example~\ref{eg!veronese}).
\end{enumerate}
\end{enumerate}

\subsection{Nonexistence and other challenges}\label{sec!challenges}
The \grdb\ is simply one way of assembling and presenting the vast amount
of data associated to the classification of Fano~\tfs, and as such it
naturally invites more questions than it answers.
A selection of topics:
\begin{enumerate}
\item
Find~$P\in\Fss$ not realised by a Fano~\tf, or not realised by a prime Fano~\tf.
\item
Can one show that each model for~$P\in\Fss$ in codimension~$5$ or~$6$ is realised by a Fano~\tf\
using similar birational methods as in codimension~$4$?
\item
There are toric Fano~\tfs\ in high codimension: can projection from these
realise Fano~\tfs\ for sequences of Hilbert series?
\item
Fano~\tfs\ with~$|{-}K_X|$ empty are rare.
The \grdb\ has~$264$ semistable Hilbert series with linear coefficient zero
(the left-hand column~$g=-2$ of Table~\ref{fig!geography}).
The first are Iano--Fletcher's example~$X_{12,14}\subset\PP(2,3,4,5,6,7)$
and the three families in codimension~$4$, studied by~\cite{AR,papadakis08,CD,Taylor},
though they remain to be fully understood.
The next candidate lies in codimension~6, namely $X\subset\PP(2,3,4^2,5^2,6^2,7^2)$,
and has no Type~I or Type~II projection attack.
\item
The \grdb\ is constructed with Gorenstein projection in mind,
and the Fano~\tf\ database includes that data; click~\cite{grdbweb}.
Sarkisov links provide another connection between Fano~\tfs, and projection is often the
first step of a Sarkisov link.
Classification attacks often exploit such links,~\cite{takeuchi} for example, but
does it even make sense to describe a web of Sarkisov links overlaying the \grdb? (cf. \cite{BBK}.)
\item
The bounds of~\cite{kawamata} work over any field~$k$ of characteristic~$0$, not necessarily
algebraically closed.
Thus, for example, the \grdb\ makes sense over~$k=\Q$
(the Type~I Ansatz \ref{typeIansatz} is defined over~$\Q$, but that is not clear for more complicated unprojection constructions),
and the generic fibres of relatively 3\nobreakdash-dimensional Mori fibre spaces
also have relative Hilbert series in the Fano~\tf\ database. Compare the
general conditions and classifications of \cite{CFST1,CFST2}.
\end{enumerate}

\subsection{Closing repetition of the main warning}\label{sec!warning}
It is easy to mistake the Fano~\tf\ database for a classification of prime Fano~\tfs.
It is~\emph{not} that classification: that classification does not yet exist.
It is instead the classification of
genus--basket pairs that satisfy certain conditions of geometric origin,
or equivalently it is a list of the
rational functions they determine by the plurigenus formula.

The confusion arises in part because each rational function is
presented as though it is the Hilbert series
of a Fano~\tf~$X\subset\PP(a_0,\ldots,a_n)$
embedded by its total anticanonical ring in weighted projective space
with given weights. These weights are sometimes accurate, and in all known examples they appear at least as a subset of the generators,
but they are not proven to be necessary and should be treated only as a convenient informed estimate. There is no claim that a Fano~\tf\
exists with this data, nor that a particular one you may be considering
is necessarily embedded anticanonically as indicated here.
Moreover, a single Hilbert series may be realised by more
than one family of varieties, and these multitudes may lie in different ambient
weighted projective spaces.

\subsection*{Acknowledgements}
It is our pleasure to thank Tom Coates and Alessio Corti for help and suggestions, especially around the Fanosearch programme (\S\ref{sec!fanosearch}) and to Andrea Petracci for explaining his work on toric deformations, including Example~\ref{eg!def1}. We thank John Cannon and the Magma group~\cite{magma} for many years' use of the Magma Computational Algebra package, on which the vast majority of calculations were made. The ongoing development of the~\grdb\ is supported by EPSRC: GR/S03980/01, EP/E000258/1, EP/N022513/1.

\newcommand{\etalchar}[1]{$^{#1}$}

\end{document}